\documentclass[a4paper, 11pt]{article}
\usepackage{amsmath, amssymb, amsthm, latexsym, enumerate}
\usepackage[alphabetic]{amsrefs}
\usepackage{xspace}
\usepackage{a4wide}
\usepackage[usenames,dvipsnames,svgnames]{xcolor}
\usepackage{tikz}
\usepackage{hyperref}
\usepackage{authblk}
\usepackage{chngcntr}
\usepackage{enumitem}

\newtheorem{btheorem}{Theorem}[]
\newtheorem{bcorollary}[btheorem]{Corollary}

\newtheorem*{ctheorem}{Theorem \hyperref[bthm:key]{\ref{bthm:key}.\ref{key:quasiisom}'}}

\newtheorem{proposition}{Proposition}[section]
\newtheorem{theorem}[proposition]{Theorem}
\newtheorem{corollary}[proposition]{Corollary}

\newtheorem{lemma}[proposition]{Lemma}

\newtheorem*{theorem*}{Theorem}

\theoremstyle{definition}
\newtheorem{definition}[proposition]{Definition}

\newtheorem{remark}[proposition]{Remark}

\newcommand{\set}[1]{\left\{#1\right\}}
\newcommand{\setcon}[2]{\left\{#1\ \left|\ #2\right.\right\}}

\newcommand{\abs}[1]{\left\lvert#1\right\rvert}
\newcommand{\into}{\hookrightarrow}

\newcommand{\R}{\mathbb{R}}
\newcommand{\Z}{\mathbb{Z}}
\newcommand{\N}{\mathbb{N}}

\newcommand{\fpres}[2]{\left\langle #1 \left| #2 \right.\right\rangle}

\newcommand{\cdim}[1]{\textrm{cdim}(#1)}

\newcommand{\Mcd}{\textrm{cdim}_{\partial_M}}
\newcommand{\stabdim}{\textrm{asdim}_s}

\newcommand{\Ngp}[2]{\left(#1\cdot_N #2 \right)_e}
\newcommand{\Npgp}[2]{\left(#1\cdot_{N'} #2 \right)_e}
\newcommand{\Npgpq}[2]{\left(#1\cdot_{N'} #2 \right)_{q(e)}}

\newcommand{\Mod}{\mathrm{Mod}(\Sigma)}
\newcommand{\teich}{\mathcal{T}(\Sigma)}

\title{Stability and the Morse boundary}

\author[1]{Matthew Cordes\thanks{The author was partially supported by NSF grant DMS-1106726}}

\author[2]{David Hume\thanks{During this work the author was supported by the ERC grant
no.\! 278469 and the grant ANR-14-CE25-0004 ``GAMME''}}

\affil[1]{Department of Mathematics, The Technion, Haifa, Israel}
\affil[2]{Mathematical Institute, University of Oxford, Woodstock Road, Oxford, UK}

\date{\today}

\begin{document}

\maketitle

\begin{abstract}
Stable subgroups and the Morse boundary are two systematic approaches to collect and study the hyperbolic aspects of finitely generated groups. In this paper we unify and generalise these strategies by viewing any geodesic metric space as a countable union of stable subspaces: we show that every stable subgroup is a quasi--convex subset of a set in this collection and that the Morse boundary is recovered as the direct limit of the usual Gromov boundaries of these hyperbolic subspaces.

We use this approach, together with results of Leininger--Schleimer, to deduce that there is no purely geometric obstruction to the existence of a non-virtually--free convex cocompact subgroup of a mapping class group.

In addition, we define two new quasi--isometry invariant notions of dimension: the stable dimension, which measures the maximal asymptotic dimension of a stable subset; and the Morse capacity dimension, which naturally generalises Buyalo's capacity dimension for boundaries of hyperbolic spaces.

We prove that every stable subset of a right--angled Artin group is quasi--isometric to a tree; and that the stable dimension of a mapping class group is bounded from above by a multiple of the complexity of the surface. In the case of relatively hyperbolic groups we show that finite stable dimension is inherited from peripheral subgroups.

Finally, we show that all classical small cancellation groups and certain graphical small cancellation groups - including some Gromov monster groups - have stable dimension at most $2$.
\end{abstract}

\section{Introduction}
A subset $Y$ of a geodesic metric space $X$ is \textbf{Morse} if for every $K\geq 1,C\geq 0$ there is some $N=N(K,C)$ such that every $(K,C)$-quasi--geodesic with endpoints on $Y$ is contained in the $N$--neighborhood of $Y$. We call the function $N$ a Morse gauge. This definition has its roots in a classical paper of Morse \cite{Morse}.

Morse geodesics and quasi--geodesics are a recurring theme in the study of groups admitting some sort of large--scale negative curvature.

The Morse lemma, which states that every quasi--geodesic in a hyperbolic space is Morse, is a vital ingredient in the proof that hyperbolicity is preserved by quasi--isometry. Moreover, any point in an asymptotic cone which is contained in the limit of a Morse quasi--geodesic is necessarily a cut--point. This plays a crucial role in the proof that metric relative hyperbolicity is also preserved by quasi--isometry \cite{drutu-sapir}. More generally, every acylindrically hyperbolic group has Morse bi-infinite geodesics \cite{sisto:2016aa}. This class of groups, recently unified by Osin in \cite{OsinAclynhyp}, encompasses many groups of significant interest in geometric group theory: hyperbolic and relatively hyperbolic groups, non-directly decomposable right-angled Artin groups, mapping class groups, and $\mathrm{Out}(F_n)$.

Morse geodesics in geodesic metric spaces can be classified by their contraction and divergence properties \cites{ACGH1,  Drutu-Mozes-Sapir, Aougab-Durham-Taylor}. Furthermore, there are classifications of Morse quasi--geodesics for CAT($0$)--spaces \cite{charney-sultan}, relatively hyperbolic groups and spaces \cite{drutu-sapir,OsinRelhyp} and graphical small cancellation groups \cite{ACGH2}.

There have been several attempts at understanding groups and spaces by considering subgroups/subspaces which are strongly quasi--convex in the ambient group/space. 

We say that a quasi--convex subspace $Y$ of a geodesic metric space $X$ is $N$--\textbf{stable} if every pair of points in $Y$ can be connected by a geodesic which is $N$--Morse in $X$. We say that a subgroup is stable if it is stable as a subspace. It is important to note that this is not the original definition of stability given in \cite{Durham-Taylor}. Our definition detects the same collection of stable subsets up to quasi--isometry, and the two definitions coincide for subgroups of finitely generated groups (cf.\! Lemma \ref{lem:stable-equiv}).

Durham--Taylor prove that the collection of stable subgroups of mapping class groups are precisely those which are convex--cocompact in the sense of Farb--Mosher \cites{Farb-Mosher, Durham-Taylor}. Convex--cocompact subgroups of the mapping class group are all purely pseudo--Anosov, i.e., all infinite order elements are pseudo--Anosov.  Farb--Mosher show if there is a purely pseudo--Anosov subgroup of $\Mod$ that is not convex cocompact, then this subgroup would be a counterexample to Gromov's conjecture that every group with a finite Eilenberg--Mac Lane space and no Baumslag--Solitar subgroups is hyperbolic \cite{BestvinaQList}. 

In an effort to identify examples of convex--cocompact subgroups of the mapping class group, Koberda, Manghas, and Taylor classify the stable subgroups of all right--angled Artin groups \cite{koberda:2014}. They show that these subgroups are always free.

\medskip

A very different approach to collecting Morse geodesics is to consider a boundary of Morse directions. The Morse boundary of a proper geodesic metric space was introduced by the first author in \cite{Cordes15}, generalising the construction of the contracting boundary of a CAT($0$)--space by Charney--Sultan \cite{charney-sultan}. Both boundaries consist of equivalence classes of geodesic rays that travel in ``hyperbolic directions'' in the space (where one ray is equivalent to another if they fellow travel).

The purpose of both constructions is to generalise the Gromov boundary of a hyperbolic space \cite{Gromov87} to other classes of spaces in a quasi--isometrically rigid way; that is, a quasi--isometry of spaces defines a homeomorphism between their boundaries. In general, this is not the case for the usual visual boundary of a CAT($0$)--space as a striking example of Croke--Kleiner proves \cite{CroKle}.

However, Gromov boundaries of hyperbolic spaces enjoy a much stronger form of rigidity. They are metrisable and a quasi--isometry of spaces yields a quasi--symmetry of their boundaries. Even in simple cases the Morse boundary is not metrisable: for example, there is no countable local basis for the topology on the Morse boundary of $(\Z \ast \Z^2)$, in particular it is not metrisable \cite{Murray}.

\medskip
Our goal in this paper is to unite and generalise both these approaches by viewing any geodesic metric space as a union of stable subsets. These subsets will be indexed by Morse gauges $N$ and will be hyperbolic, with hyperbolicity constant depending only on $N$.

Given a geodesic metric space $X$, a point $e\in X$, and a Morse gauge $N$ we define $X^{(N)}_e$ to be the set of all points in $X$ which can be connected to $e$ by an $N$--Morse geodesic.

Our key technical result is the following.

\begin{btheorem}\label{bthm:key} Let $X,Y$ be geodesic metric spaces and let $e\in X$. The family of subsets $X^{(N)}_e$ of $X$ indexed by functions $N:\N^2\to\N$ enjoys the following properties:
\begin{enumerate}[label=\textup{\Roman*}]
\item (Covering) $X=\bigcup_N X^{(N)}_e$.\label{key:covering}
\item (Partial order) If $N\leq N'$, then $X^{(N)}_e\subseteq X^{(N')}_e$. \label{key:PO}
\item (Hyperbolicity) Each $X^{(N)}_e$ is hyperbolic.\label{key:hyperbolic}
\item (Stability) Each $X^{(N)}_e$ is $N'$--stable, where $N'$ depends only on $N$.\label{key:stable}
\item (Universality) Every stable subset of $X$ is a quasi--convex subset of some $X^{(N)}_e$.\label{key:universal}
\item (Boundary) The sequential boundary $\partial_s X^{(N)}_e$ can be equipped with a visual metric which is unique up to quasi--symmetry. An inclusion $X^{(N)}_e\subseteq X^{(N')}_e$ induces a map $\partial_s X^{(N)}_e\to \partial_s X^{(N')}_e$ which is a quasi--symmetry onto its image.\label{key:boundary}
\item (Generalising the Gromov boundary) If $X$ is hyperbolic, then $X=X^{(N)}_e$ for all $N$ sufficiently large, and $\partial_s X^{(N)}_e$ is quasi--symmetric to the Gromov boundary of $X$. \label{key:Gromov}
\item (Generalising the Morse boundary) If $X$ is proper, then its Morse boundary is equal to the direct limit of the $\partial_s X^{(N)}_e$ as topological spaces.\label{key:Mboundary}
\item (Behaviour under quasi--isometry) If $q:X\to Y$ is a quasi--isometry then for every $N$ there exists an $N'$ such that $q(X^{(N)}_e)\subseteq Y^{(N')}_{q(e)}$ and there is an induced embedding $\partial q:\partial_s X^{(N)}_e\to \partial_s Y^{(N')}_{q(e)}$ which is a quasi--symmetry onto its image.\label{key:quasiisom}
\end{enumerate}
\end{btheorem}

This approach of passing to strata indexed by Morse gauges is necessary. Consider the boundary of all Morse rays in a geodesic metric space topologised with the Gromov product. Cashen shows that when the space is not hyperbolic, quasi--isometries do not necessarily induce homeomorphisms of this boundary \cite{Cashen16}.
\medskip

We will prove a stronger version of \hyperref[bthm:key]{\ref{bthm:key}.\ref{key:quasiisom}} for which we require some additional terminology.

Let $X,Y$ be geodesic metric spaces, let $x\in X$ and $y\in Y$. We say $X$ is \textbf{stably subsumed} by $Y$ (denoted $X\hookrightarrow_s Y$) if, for every $N$ there exists some $N'$ and a quasi--isometric embedding $X^{(N)}_x\to Y^{(N')}_y$. By \hyperref[bthm:key]{\ref{bthm:key}.\ref{key:quasiisom}}, this property is independent of the choice of $x,y$. We say $X$ and $Y$ are \textbf{stably equivalent} (denoted $X\sim_s Y$) if they stably subsume each other. It is easy to see that two spaces are stably equivalent if and only if they have the same collection of stable subsets up to quasi--isometry.

Given a geodesic metric space $X$, we will consider the collection of boundaries $\left( \partial_s X^{(N)}_e\right)$ equipped with visual metrics as the \textbf{metric Morse boundary} of $X$. 

We say that one collection of spaces $\left( A_i\right)_{i\in I}$ is \textbf{quasi--symmetrically subsumed} by another $\left(B_j\right)_{j\in J}$ (denoted $(A_i)\hookrightarrow_{qs}(B_j)$) if, for every $i$ there exists a $j$ and an embedding $A_i \to B_j$ which is a quasi--symmetry onto its image. Two collections are \textbf{quasi--symmetrically equivalent} (denoted $(A_i)\sim_{qs}(B_j)$) if $(A_i)\hookrightarrow_{qs}(B_j)$ and $(A_i)\hookrightarrow_{qs}(A_i)$).

\begin{ctheorem} \label{bthm:keyquasiisom2} Let $X,Y$ be geodesic metric spaces, let $x\in X$ and $y\in Y$. Then $X\hookrightarrow_s Y$ if and only if $\left( \partial X^{(N)}_x\right)\hookrightarrow_{qs} \left( \partial Y^{(N)}_y\right)$.
\end{ctheorem}
The forward implication is our goal. The reverse implication essentially follows from techniques introduced in \cite{BonkSchramm}.

\begin{bcorollary}\label{bcor:stableinvarience} Quasi--isometric geodesic metric spaces are stably equivalent and have quasi--symmetrically equivalent metric Morse boundaries. In particular, the metric Morse boundary is invariant under change of basepoint.
\end{bcorollary}

Stable equivalence is a much weaker notion than quasi--isometry: virtually solvable groups, and more generally groups satisfying a non-trivial law have no Morse geodesic rays \cite{drutu-sapir}, consequently they are stably equivalent to a point. We will present more interesting examples later in the paper.
\medskip

Using the above theorem we can define two stable equivalence invariants for geodesic metric spaces. We define the \textbf{stable asymptotic dimension} of $X$ ($\stabdim (X)$) to be the maximal asymptotic dimension of a stable subset of $X$, which by universality, is the maximal asymptotic dimension of the $X^{(N)}_e$. One obvious but useful bound is that the stable asymptotic dimension is bounded from above by the asymptotic dimension.

Similarly, we define the \textbf{Morse capacity dimension} of $X$ ($\Mcd(X)$) to be the maximal capacity dimension of spaces in the metric Morse boundary. This is clearly an invariant of equivalent metric Morse boundaries. To make the next result easier to state we adopt the convention that the empty set has capacity dimension $-1$.

\begin{bcorollary}\label{bcor:comparedims} Let $X$ be a geodesic metric space. Then
\[
 \stabdim (X)-1\leq \Mcd(X) \leq \stabdim (X).
\] 
\end{bcorollary}

This follows from bounds proved in the hyperbolic setting: \cite[Theorem $1.1$]{Buycapdim} and \cite[Proposition $3.6$]{MackSis}.

By Theorem \hyperref[bthm:key]{\ref{bthm:key}.\ref{key:Gromov}}, the stable dimension of a hyperbolic space is precisely its asymptotic dimension and the Morse capacity dimension of a hyperbolic space is the capacity dimension of its boundary equipped with some visual metric.
\medskip

We note that Gruber has a method of constructing finitely generated groups which admit an infinite family of stable subgroups with unbounded asymptotic dimension, consequently they have infinite stable dimension  \cite{Gruber_personal}. We are not aware of any other examples or constructions of finitely generated groups with infinite stable dimension.
\medskip

The remainder of the paper is devoted to studying these new notions in the context of finitely generated groups in which Morse geodesics have been classified.
\medskip

We begin with mapping class groups, throughout the paper we will only consider orientable surfaces.

\begin{btheorem} \label{thm:Mod} Let $\Sigma$ be a surface of genus $g$ with $p$ punctures. Let $X$ be a Cayley graph of the mapping class group $\mathrm{Mod}(\Sigma)$ and let $\teich$ be the Teichm\"{u}ller space of $\Sigma$. The spaces $X$ and $\teich$ are stably equivalent.

Moreover, each $X^{(N)}_e$ quasi--isometrically embeds into the curve complex $\mathcal{C}(\Sigma)$. Consequently,
\begin{align*}
 \stabdim(X)=\stabdim(\teich) &\leq 4g+p-3 \text{ if } p>0 \\
	 &\leq 4g-4 \text{ if } p=0.
\end{align*}
\end{btheorem}

An upper bound on the stable dimension of mapping class groups can be obtained via the bounds on asymptotic dimension for mapping class groups obtained by Bestvina--Bromberg--Fujiwara in \cite{BBF} or by Behrstock--Hagen--Sisto in \cite{BHS} which are exponential or quadratic in the Euler characteristic of the surface respectively. Here we show that each $\Mod^{(N)}_e$ quasi--isometrically embeds into the curve graph and use the bound found by Bestvina--Bromberg on the asymptotic dimension of the curve graph \cite{Bestvina-Bromberg}.

Leininger and Schleimer prove that for every $n$ there is a surface $\Sigma$ such that the Teichm\"{u}ller space $\teich$ contains a stable subset quasi--isometric to $\mathbb{H}^n$. This gives a lower bound on the stable dimension, but it is at best logarithmic in the complexity \cite{leininger:2014aa}. 

Since $\teich$ is stably equivalent to $\Mod$, we see that $\Mod$ contains a stable subset quasi--isometric to $\mathbb{H}^n$. The only known explicit examples of convex cocompact groups are virtually free groups \cite{Dowdall:2014ab, Kent:2009aa, koberda:2014, Min:2011aa}. Although the results of Leininger--Schleimer do not provide any non-virtually--free convex cocompact subgroups, the fact that $\stabdim(\mathrm{Mod} (\Sigma))>1$ for some surfaces shows that there is no purely geometric obstruction to the existence of a non-free convex cocompact subgroup of $\Mod$.
\medskip

In an effort to identify examples of convex--cocompact subgroups of the mapping class group, Koberda, Manghas, and Taylor classify the stable subgroups of all right--angled Artin groups \cite{koberda:2014}. They show that these subgroups are always free. Here we prove the natural analogue for stable subspaces.

\begin{btheorem} \label{thm:RAAG} Let $X$ be a Cayley graph of a right--angled Artin group. Every stable subset of $X$ is quasi--isometric to a proper tree. In particular, $X$ is stably equivalent to a line if the group is abelian of rank $1$, a point if it is abelian of rank $\neq 1$ and a regular trivalent tree otherwise.
\end{btheorem}

To prove this we will show that each $X^{(N)}_e$ quasi--isometrically embeds into the contact graph defined by Hagen \cite{hagen:2014aa}. As a result, each $X^{(N)}_e$ is quasi--isometric to a proper tree; to complete the proof we use the universality condition in Theorem \ref{bthm:key}.  By universality of stable subsets (Theorem \hyperref[bthm:key]{\ref{bthm:key}.\ref{key:universal}}), we know that any stable subgroup of a right--angled Artin group is quasi--isometric to a tree. Thus since groups which are quasi--isometric to trees are virtually free  \cite[Corollary 7.19]{GhysdlH}, we recover one result of \cite{koberda:2014}: stable subgroups of right--angled Artin groups are free.  Using this same map to the contact graph, we also classify when quasi--convex subsets of right--angled Artin groups are stable:

\begin{btheorem} Let $\Gamma$ be a finite graph, let $A_\Gamma$ be the corresponding right--angled Artin group. Let $\tilde{S}_\Gamma$ be the universal cover of the Salvetti complex and let $\mathcal{CG}_\Gamma$ be the contact graph of $\tilde{S}_\Gamma$. There is a map $q:A_\Gamma\to\mathcal{CG}_\Gamma$ such that if $Y\subseteq A_\Gamma$ is quasi--convex, then $Y$ is stable in $A_\Gamma$ if and only if $q|_Y$ is a quasi--isometric embedding.
\end{btheorem}

Morse geodesics in graphical small cancellation groups were classified in \cite{ACGH2}. Classical small cancellation theory covers a variety of techniques in combinatorial group theory to ensure that a presentation yields a group which is large in some sense; or easily understandable, for instance with solvable word and conjugacy problems. An excellent introduction to the theory is the book of Lyndon--Schupp \cite{LS01}. Graphical small cancellation theory is a generalization introduced by Gromov \cite{Gromov03} in order to construct groups whose Cayley graphs contain certain prescribed subgraphs, in particular one can construct ``Gromov monster'' groups, those with a Cayley graph which coarsely contains expanders \cite{Arzh-Delz, Osajda}. These monster groups cannot be coarsely embedded into a Hilbert space, and they are the only known counterexamples to the Baum--Connes conjecture with coefficients \cite{HigLafSka02}.

\begin{btheorem}\label{bthm:smallcanc} Let $X$ be the Cayley graph of a classical $C'(1/6)$ small cancellation group. Then $\stabdim (X)\leq 2$ and $\Mcd(X)\leq 1$.
\end{btheorem}

Note that this is optimal as fundamental groups of higher genus surfaces are hyperbolic with asymptotic dimension $2$ and admit $C'(\frac{1}{6})$ graphical small cancellation presentations.

Again we work with the spaces $X^{(N)}_e$. Each of these embeds quasi--isometrically into a finitely presented classical $C'(1/6)$ small cancellation group. These are hyperbolic with asymptotic dimension at most $2$ and the capacity dimension of their Gromov boundaries is at most $1$.

Our method in proving the above theorem allows for some interesting generalisations.

\begin{btheorem}\label{bthm:Gromovmonster} There exist graphical small cancellation groups with stable dimension at most $2$ admitting a Cayley graph which isometrically contains an expander.
\end{btheorem}

The final goal is to show that, as for asymptotic dimension \cite{OsinRelhyp}, relatively hyperbolic group inherit finite stable dimension from their maximal parabolic subgroups.

\begin{btheorem}\label{bthm:relhyp} Let $G$ be a finitely generated group which is hyperbolic relative to $\mathcal H$. Then $\stabdim (G)<\infty$ if and only if $\stabdim (H)<\infty$ for all $H\in\mathcal{H}$.
\end{btheorem}

Many classes of groups are already known to have finite stable dimension: virtually solvable groups---and more generally any unconstricted groups, for instance those satisfying a non-trivial law---have no infinite Morse rays \cite{drutu-sapir}, and any group with finite asymptotic dimension. As yet there is no example of an amenable (not virtually cyclic) group admitting a Morse ray. 
\medskip

Variations of the constructions in this paper are ideally suited to the study of relatively hyperbolic groups. These constructions will be the focus of a future paper \cite{CordesHume_relativehyp}; the main result of which states that given any finite collection of non-relatively hyperbolic groups $\mathcal{H}$, there are infinitely many $1$--ended groups which are hyperbolic relative to $\mathcal{H}$. In the specific case where each $H\in\mathcal{H}$ has finite stable dimension, one can find a family of such relatively hyperbolic groups with unbounded stable dimension.

\subsection*{Acknowledgements}

The authors would like to thank Jason Behrstock, Ruth Charney, Matthew Gentry Durham and Dominik Gruber for interesting conversations. We are also grateful to the referee for many comments which improved the clarity of the paper.

\section{Preliminaries}

\subsection{Morse geodesics and the Morse boundary}

\begin{definition}[Quasi--geodesics] A map $f:(X,d_X)\to (Y,d_Y)$ is a \textbf{quasi--isometric embedding} if there exist constants $K\geq 1$ and $C\geq 0$ such that
\[
 K^{-1}d_X(x,y) - C \leq d_Y(f(x),f(y)) \leq K d_X(x,y) + C
\]
holds for all $x,y\in X$. In this case we call $f$ a $(K,C)$-quasi--isometric embedding. If $X$ is a connected subset of $\R$ then we call $f$ a $(K,C)$-\textbf{quasi--geodesic}.
\end{definition}
When we describe a quasi--geodesic as a subset of a space, we are implicitly referring to the image of the map.

\begin{definition} [Morse geodesics] A geodesic $\gamma$ in a metric space is called \textbf{Morse} if there exists a function $N=N(K, C)$ such that for any $(K, C)$-quasi--geodesic $\sigma$ with endpoints on $\gamma$, we have that $\sigma \subset \mathcal{N}_N(\gamma)$, the $N$-neighborhood of $\gamma$. We call the function $N$ a \bf{Morse gauge} and say that $\gamma$ is $N$--Morse.
\end{definition}
	
\begin{definition}  \textbf{The Morse boundary}

Let $X$ be a proper geodesic space and fix a basepoint $p \in X$. The \emph{Morse boundary} of $X$, $\partial_M X$, is the set of all Morse geodesic rays in $X$ (with basepoint $p$) up to asymptotic equivalence. To topologise the boundary, first fix a Morse gauge $N$ and consider the subset of the Morse boundary that consists of all rays in $X$ with Morse gauge at most $N$:  \begin{equation*} \partial_M^N X_p= \{[\alpha] \mid \exists \beta \in [\alpha] \text{ that is an $N$--Morse geodesic ray with } \beta(0)=p\}. \end{equation*} We topologise this set with the compact-open topology. This topology is equivalent to one defined by a system of neighbourhoods, $\{V_n(\alpha) \mid n \in \N \}$, at a point $\alpha$ in $\partial_M^N X_p$. The sets $V_n( \alpha)$ are defined to be the set of geodesic rays $\gamma$ with basepoint $p$ and $d(\alpha(t), \gamma(t))< \delta_N$ for all $t<n$ where $\delta_N$ is a constant that depends only on $N$. 

Let $\mathcal M$ be the set of all Morse gauges. We put a partial ordering on $\mathcal M$ so that  for two Morse gauges $N, N' \in \mathcal M$, we say $N \leq N'$ if and only if $N(K,C) \leq N'(K,C)$ for all $K,C$. We define the Morse boundary of $X$ to be
 \begin{equation*} \partial_M X_p=\varinjlim_\mathcal{M} \partial^N_M X_p \end{equation*} with the induced direct limit topology, i.e., a set $U$ is open in $\partial_M X_p$ if and only if $U \cap \partial^N_M X_p$ is open for all $N$.   

\end{definition}
\begin{theorem}[\cite{Cordes15}] Given a proper geodesic space $X$, the Morse boundary, $\partial_M X=\varinjlim \partial_M^{N} X_p$, equipped with the direct limit topology, is 
\begin{enumerate}
\item basepoint independent;
\item a quasi--isometry invariant; and
\item homeomorphic to the visual boundary if $X$ is hyperbolic and to the contracting boundary if $X$ is $\mathrm{CAT}(0)$.
\end{enumerate}
\end{theorem}

For more details on the Morse boundary see \cite{Cordes15}.

\subsection{Boundaries of hyperbolic spaces}\label{sec:visbdry}

Here we will quickly recall the construction of a sequential boundary with a visual metric for a hyperbolic space, this can all be found in \cite[Chapter III.H]{bh} or \cite{GhysdlH}.

\begin{definition}\label{defn:hyp} \cite[Definition $1.20$]{bh} Let $X$ be a (not necessarily geodesic) metric space. We say $X$ is $\delta$\textbf{--hyperbolic} if for all $w,x,y,z$ we have
\[
 (x\cdot y)_w \geq \min\set{(x\cdot z)_w,(z\cdot y)_w} - \delta.
\]
\end{definition}

\begin{definition} Let $X$ be a $\delta$--hyperbolic metric space and let $x,y,z\in X$. The \textbf{Gromov product} of $x$ and $y$ with respect to $z$ is defined as
\[
 (x\cdot y)_z = \frac{1}{2}\left( d(z,x)+d(z,y)-d(x,y)\right).
\]
Let $(x_n)$ be a sequence in $X$. We say $(x_i)$ converges at infinity if $(x_i\cdot x_j)_e \to \infty$ as $i,j \to \infty$. Two convergent sequences $(x_n),(y_m)$ are said to be \textbf{equivalent} if $(x_i \cdot y_j) \to \infty$ as $i,j \to \infty$. We denote the equivalence class of $(x_n)$ by $\lim x_n$.

The \textbf{sequential boundary} of $X$, $\partial_s X$ is defined to be the set of convergent sequences considered up to equivalence.
\end{definition}
As all boundaries of hyperbolic spaces considered in this paper will be sequential, we drop the subscript $s$ from the notation. We may extend the Gromov product to $\partial  X$ in the following way:
\[
 (x\cdot y)_e = \sup \left(\liminf_{m,n\to\infty}\set{(x_n\cdot y_m)_e}\right).
\]
where the supremum is taken over all sequences $(x_i)$ and $(y_j)$ in $X$ such that $x= \lim x_i$ and $y=\lim y_j$.

This Gromov product satisfies the following key properties:

\begin{lemma} \label{lem:gromovproduct} Let $X$ be a $\delta$--hyperbolic metric space and fix $e \in X$. 
\begin{enumerate} 
\item \label{lem:equivmeansinfinite} $(x \cdot y)_e = \infty$ if and only if $x=y \in \partial  X$.
\item \label{lem:canpickconvergentseq} For all $x,y \in \partial  X\cup X$ there exist sequences $(x_n)$ and $(y_n)$ in $X$ such that $x = \lim x_n$, $y=\lim y_n$ and $(x\cdot y)_e =\lim_n (x_n \cdot y_n)_e$.
\item \label{lem:uptoreplacement} For all $x,y \in \partial  X$ and all sequences $(x_i')$ and $(y_j')$ in $X$ with $x = \lim x_i'$ and $y= \lim y_j'$, $$(x\cdot y)_e-2\delta \leq \liminf_{i,j }(x'_i \cdot y_j')_e \leq (x\cdot y)_e.$$
\item \label{lem:seqroughtriangle} For all $x, y ,z \in \partial  X\cup X$ we have $(x\cdot z)_e \geq \min\{(x\cdot y)_e, (z\cdot y)_e \} - 2\delta$. 
\end{enumerate} 
\end{lemma}

Recall that a metric $d$ on $\partial  X$ is said to be \textbf{visual} (with parameter $\varepsilon>0$) if there exist $k_1,k_2>0$ such that $k_1\exp(-\varepsilon(x\cdot y)_e) \leq d(x,y)\leq k_2\exp(-\varepsilon(x\cdot y)_e)$, for all $x,y\in\partial X$.

Let $x, y \in \partial  X$. As a shorthand we define $\rho_\epsilon(x, y) := \exp\left(-\varepsilon (x\cdot y\right)_e)$.

\begin{theorem}\label{thm:GrProdmetric} \textrm{\cite[Section $7.3$]{GhysdlH}} Let $X$ be a $\delta$--hyperbolic space. If $\varepsilon'=\exp(2\delta\varepsilon)-1\leq \sqrt{2}-1$ then we can construct a visual metric $d$ on $\partial  X$ such that
\[
 (1-2\varepsilon')\rho_\varepsilon(x, x')\leq d(x,x') \leq \rho_\varepsilon(x,x').
\]
\end{theorem}

Visual metrics on a hyperbolic space are all quasi--symmetric.

\begin{definition}\label{defn:quasisym} A homeomorphism $f:(X,d)\to (Y,d')$ is said to be quasi--symmetric if there exists a homeomorphism $\eta:\R\to\R$ such that for all distinct $x,y,z\in X$, 
\[
 \frac{\displaystyle d'(f(x),f(y))}{\displaystyle d'(f(x),f(z))} \leq \eta \left( \frac{\displaystyle d(x,y)}{\displaystyle d(x,z)}\right).
\]
\end{definition}

One natural invariant of the boundary of a hyperbolic space is its \textbf{capacity dimension}, introduced by Buyalo in \cite{Buycapdim}.

Let $\mathcal{U}$ be an open covering of a metric space $X$. Given $x \in X$, we let $$L(\mathcal{U},x)=\sup\setcon{d(x, X\char`\\U)}{U \in \mathcal{U}}$$ be the {\bf Lebesgue number of $\mathcal{U}$ at $x$ } and $L(\mathcal{U}) = \inf_{x \in X} L(\mathcal{U},x)$ the Lebesgue number of $\mathcal{U}$. The {\bf multiplicity of $\mathcal{U}$}, $m(\mathcal{U})$, is the maximal number of members of $\mathcal{U}$ with non-empty intersection.

\begin{definition}[\cite{BuyBook}] The \textbf{capacity dimension} of a metric space $X$ ($cdim(X)$) is the minimal integer $m$ with the following property:

There exists some $\delta \in (0,1)$ such that for every sufficiently small $r>0$ there is an open covering $\mathcal{U}$ of $X$ by sets of diameter at most $r$ with $L(\mathcal{U}) \geq \delta r$ and $m(\mathcal{U})\leq m+1$.
\end{definition}

By Corollary $4.2$ of \cite{Buycapdim} the capacity dimension of a metric space is a quasi--symmetry invariant.

\section{Metric Morse boundary}\label{sec:metricMorse}

The goal of this section is to prove Theorem \ref{bthm:key}.

Let $X$ be a geodesic metric space, let $e\in X$ and let $N$ be a Morse gauge. We define $X^{(N)}_e$ to be the set of all $y\in X$ such that there exists a $N$--Morse geodesic $[e,y]$ in $X$.

Theorem \hyperref[bthm:key]{\ref{bthm:key}.\ref{key:covering}} is trivial.

\begin{lemma}\label{lem:covering} Let $x\in X$. There exists some $N$ such that $x\in X^{(N)}_e$.
\end{lemma}
\begin{proof} Let $d=d(e,x)$ and define $N(K,C)=Kd+C$. Any geodesic from $e$ to $x$ is $N$--Morse, since any $(K,C)$-quasi--geodesic from $e$ to $x$ is contained in the closed ball of radius $Kd+C$ centred at $e$.
\end{proof}

Theorem \hyperref[bthm:key]{\ref{bthm:key}.\ref{key:PO}} follows immediately from the definition of the spaces $X^{(N)}_e$.

\subsection{Hyperbolicity and Stability (\hyperref[bthm:key]{\ref{bthm:key}.\ref{key:hyperbolic}} and \hyperref[bthm:key]{\ref{bthm:key}.\ref{key:stable}})}

Our first goal is the following proposition.

\begin{proposition}\label{prop:hypsubsets} $X^{(N)}_e$ is $8N(3,0)$--hyperbolic in the sense of Definition $\ref{defn:hyp}$.
\end{proposition}

We will prove this as a consequence of the next three lemmas, the latter two will also be useful in the rest of the section.

\begin{lemma}\label{lem:Morserayshyp} Let $(X,d)$ be a geodesic metric space and let $\ell_1,\ell_2$ be two geodesics in $X$ with $\ell_1(0)=\ell_2(0)=e$.  Let $x_i\in \ell_i$, let $P$ be a geodesic from $x_1$ to $x_2$ such that $d(e,P)$ is minimal at some $z\in P$ and let $\gamma$ be any geodesic from $z$ to $e$. Then the concatenations $x_1\to_P z\to_\gamma e$ and $x_2\to_P z\to_\gamma e$ are $(3,0)$-quasi--geodesics.  

If $\ell_1,\ell_2$ are $N$--Morse then $d(z,\ell_i)$ is uniformly bounded by $N(3,0)$.
\end{lemma}
\begin{proof}
Let $w\in \gamma$ with $d(w,z)=s_1$ and $w'\in P(x_1 \to z)$ with $d(w',z)=s_2$. It suffices to show that $d(w,w')\geq \frac{1}{3}(s_1+s_2)$.

Firstly, notice that $d(w,w')\geq s_1$, as $z$ is a closest point to $e$ on $P$. Secondly, $d(w,w')\geq s_2-s_1$ by the triangle inequality.

Combining these we see that if $2s_1\geq s_2$ then $d(w,w')\geq s_1\geq \frac{1}{3}(s_1+s_2)$, and if $2s_1\leq s_2$ then $d(w,w')\geq s_2-s_1\geq \frac{1}{3}(s_1+s_2)$.
\end{proof}

\begin{lemma}\label{lem:geodlinspeed} Let $\ell_1,\ell_2$ be two $N$--Morse geodesics with $\ell_1(0)=\ell_2(0)=e$. Suppose there exists some $t_0$ so that $d(\ell_1(t),\ell_2(t))> 4N(3,0)$ for all $t\geq t_0$. Fix $t_0$ with this property such that $d(\ell_1(t_0),\ell_2(t_0))\leq 6N(3,0)$.

Let $t_1,t_2\geq t_0$ and define $x_i=\ell_i(t_0)$, $y_i=\ell_i(t_i)$. Let $P$ be the path obtained from following $\ell_1$ from $y_1$ to $x_1$, taking any geodesic $\ell$ to $x_2$ then taking $\ell_2$ to $y_2$. The path $P$ is a $(1,12N(3,0))$ quasi--geodesic. In particular, we have
\[
 d=d(y_1,y_2)\geq d(y_1,x_1)+ d(x_1,x_2)+d(x_2,y_2) - 12N(3,0).
\]
\end{lemma}
\begin{proof} For brevity set $N=N(3,0)$. By Lemma \ref{lem:Morserayshyp} every geodesic from $y_1$ to $y_2$ is contained in the $N$ neighbourhood of $\ell_1\cup \ell_2$. Fix such a geodesic $\gamma:[0,d]\to X$ with $\gamma(0)=y_1$, $\gamma(d)=y_2$. There exists a point $z\in \gamma$ such that $d(z,\ell_1(s)),d(z,\ell_2(s'))\leq N$. Hence
\[ \max_{y=s,s'}d(\ell_1(y),\ell_2(y)) \leq d(\ell_1(s),\ell_2(s')) +  \abs{s-s'} \leq 2d(\ell_1(s),\ell_2(s')) \leq 4N,
\]
so $s\leq t_0$. Similarly, $s'\leq t_0$.
 
If we choose $r_1\in[0,d]$ maximal and $r_2\in[0,d]$ minimal such that $d(\gamma(r_i),\ell_i)\leq N$ we see that $d(\gamma(r_i),e)\leq t_0+N$. By assumption $\gamma([0,r_2])$ and $\gamma([r_1,d])$ are contained in the $N$ neighbourhoods of $\ell_1,\ell_2$ respectively. Moreover, there exists some $z_1\in \gamma([0,r_2]$, $z_2\in \gamma([r_1,d])$ with $d(z_i,e)=t_0+N$. Set $x_i=\ell_i(t_0)$. We claim that $d(x_i,z_i)<3N$. To verify this, let $v_i$ be the closest point on $\ell_i$ to $z_i$. Now $d(v_i,z_i)<N$ so $d(e,v_i)\in (t_0,t_0+2N)$, thus $d(x_i,z_i)\leq d(x_i,v_i)+d(v_i,z_i)<3N$.

Hence, $d(\gamma,\ell_i(t_0))\leq 3N$. We now verify that the path $P$ is a $(1,12N)$-quasi--geodesic. Let $a,b\in P$. Firstly, suppose $a\in\ell_1[x_1,y_1]$ and $b\in\ell_2[x_2,y_2]$. Then any geodesic $\gamma'$ from $a$ to $b$ contains points $u_i$ such that $d(u_i,\ell_i(t_0))\leq 3N$ by the above argument. It follows that
\[
 \abs{P[a,b]}= d(a,x_1)+d(x_1,x_2)+d(x_2,b) \leq d(a,u_1)+d(u_1,u_2) + d(u_2,b) + 12N \leq d(a,b)+12N.
\]
Secondly, suppose $a\in \ell_1[x_1,y_1]$ and $b\in \ell$. By the above, every geodesic from $a$ to $x_2$ contains a point $u_1$ with $d(u_1,x_1)\leq 3N$. Now
\[
 d(a,x_1) + d(x_1,b) +d(b,x_2) \leq d(a,u_1)+d(u_1,x_2) + 6N \leq d(a,b)+d(b,x_2) \leq d(a,b)+12N.
\]
\end{proof}

Recall that a geodesic triangle $\gamma_1\cup\gamma_2\cup\gamma_3$ is said to be \textbf{slim} if there is a constant $D$ such that each side of the triangle is contained in the union of the $D$--neighbourhoods of the other two sides. In this case we say the triangle is $D$--slim.

\begin{lemma} \label{lem:deltaslim} Let $\ell_1$ and $\ell_2$ be $N$--Morse geodesics such that $\ell_1(0)=\ell(0)=e$ and let $x_1=\ell_1(v_1),x_2=\ell_2(v_2)$ be points on $\ell_1$ and $\ell_2$ respectively. Let $\gamma\colon [0,s] \to x$ be a geodesic between $x_1$ and $x_2$. Then the geodesic triangle $\ell_1([0,v_1]) \cup \gamma \cup \ell_2([0, v_2])$ is $4N(3,0)$-slim.
\end{lemma}

\begin{proof} Choose $\gamma(x)$ so that it is a nearest point on $\gamma$ to $e$ and let $\eta$ be a geodesic connecting $\gamma(x)$ and $e$. By Lemma \ref{lem:Morserayshyp}, the concatenation $\phi_1=\gamma([0, x]) \cup \eta$ and $\phi_2=\bar{\eta} \cup \gamma([x,s])$ are $(3,0)$-quasi--geodesics, where $\bar{\eta}$ is the geodesic $\eta$ with its orientation reversed. 

Using Lemma 2.1 from \cite{Cordes15}, we know that the Hausdorff distance between $\phi_i$ and $\ell_i$ is less than $2N(3,0)$ for $i=1,2$. Thus for every $t \in [0,v_1]$ either $d(\ell_1(t), \gamma)< 4N(3,0)$ or $d(\ell_1(t), \ell_2([0,v_2]))<4N(3,0)$. So $\ell_1 \subset \mathcal{N}_{4N(3,0)}(\ell_2 \cup \gamma)$. The final containment follows identically.
\end{proof}

\begin{proof}[Proof of Proposition $\ref{prop:hypsubsets}$] Since each geodesic triangle with vertices $e,x,y\in X^{(N)}_e$ is $4N(3,0)$--slim, \cite[III.H.1.22]{bh} states that for all $x,y,z\in X^{(N)}_e$, $(x \cdot y)_e \geq \min\set{(x\cdot z)_e,(z\cdot y)_e} - 4N(3,0)$. Then by \cite[$1.1B$]{Gromov87}, $X^{(N)}_e$ is $8N(3,0)$--hyperbolic in the sense of Definition \ref{defn:hyp}.
\end{proof}

Stability follows by applying the argument in \cite[Lemma $2.3$]{Cordes15}.

\begin{lemma} \label{lem:thirdedgeMorse} Let $X$ be a geodesic metric space with basepoint $e$. For every $N$ there exists some $N'$ such that given any $a,b \in X^{(N)}_e$, there is an $N'$--Morse geodesic from $a$ to $b$. 
\end{lemma}

The definition of stability given in the introduction is not exactly the one stated in Durham--Taylor. We now recall their definition and compare the two statements.

\begin{definition}[Stability \cite{Durham-Taylor}] Let $f \colon Y \to X$ be a quasi--isometric embedding between geodesic metric spaces. We say $Y$ is \textbf{stable} in $X$ if, for all $K,C \geq 0$ there exists an $R=R(K,C) \geq 0$ so that if $\gamma \colon [a,b]\to X$ and $\gamma' \colon [a',b']\to X$ are $(K, C)$-quasi--geodesics with $\gamma(a)=\gamma'(a') \in f(Y)$ and $\gamma(b)=\gamma'(b') \in f(Y)$, then the Hausdorff distance between $\gamma$ and $\gamma'$ is less than $R$.
\end{definition}

To obtain an equivalent statement to the definition of stability we are using we must relax the condition that $Y$ is a geodesic metric space and instead assume only that it is \textbf{quasi--geodesic}: there exist $K\geq 1$, $C\geq 0$ such that every pair of points in $Y$ can be connected by a $(K,C)$-quasi--geodesic.

\begin{lemma} \label{lem:stable-equiv} A quasi-geodesic space $Y$ is stable in a geodesic space $X$ in the sense of Durham--Taylor if and only if $f(Y)$ is $N$--stable as a subspace of $X$ for some Morse gauge $N$.
\end{lemma}

\begin{proof} Let $Y$ be a Durham--Taylor stable in $X$. Let $a,b \in f(Y)$ and let $\gamma$ be a geodesic between $a$ and $b$. By definition of Durham--Taylor stable, we know that $\gamma$ is $R$--Morse. Also, since $Y$ is quasi--isometrically embedded we know that there exists a $(K',C')$-quasi--geodesic between $a$ and $b$ in $X$ whose image lies in $Y \subset X$, where $K',C'$ are independent of the choice of $a,b$. Therefore we can also conclude that $Y$ is $R$-quasi--convex. Thus $f(Y)$ is $R$--stable as a subspace of $X$.
\medskip

Now assume that $Y$ is $N$--stable in $X$ and define $f:Y\to X$ to be the natural injection. By definition, $f$ is an isometric embedding and, since it is quasi--convex and $X$ is geodesic, $Y$ is quasi--geodesic. To see this, let $y,y'\in Y$ and let $\gamma$ be a geodesic connecting them in $X$. For each suitable $i\in\N$ set $x_i$ to be the point on $\gamma$ satisfying $d_X(y,x_i)=i$. Since $Y$ is quasi--convex, there exists a constant $B$ and points $y_i\in Y$ such that $d_X(x_i,y_i)\leq B$. The points $y=y_0,y_1,y_2,\dots, y_n,y'$ define a $(K,C)$-quasi--geodesic (where $K$ and $C$ depend only on $B$) from $y$ to $y'$ in $Y$.

Let $a, b \in Y$. Then by definition of $N$--stable we know that we can join $a$ and $b$ by a $N$--Morse geodesic $\gamma$ in $X$. Any two $(K,C)$-quasi--geodesics $\sigma_1, \sigma_2$ with endpoints $a$ and $b$ will be in the $N(K,C)$--neighbourhood of $\gamma$. Thus the $\sigma_i$ will be within Hausdorff distance $2N(K,C)$ of $\gamma$ and therefore in the $4N(K,C)$--neighbourhood of each other. Hence, $Y$ is Durham--Taylor stable.
\end{proof}

Since $f(Y)$ is quasi--isometric to $Y$, the two definitions of stability yield the same collection of stable subspaces up to quasi--isometry.

\subsection{Universality (\hyperref[bthm:key]{\ref{bthm:key}.\ref{key:universal}})}

Let $Y\subset X$ be $N$--stable and let $y\in Y$. By construction $Y\subseteq X^{(N)}_y$ so universality follows from the next lemma which we express in a far more general context as it will be useful later.

\begin{lemma} \label{lem:straightentoNprime} Let $x \in X^{(N)}_e$ and let $q\colon X \to Y$ be a $(K,C)$-quasi--isometry. Then $q(x) \in Y^{(N')}_{q(e)}$ for some $N'$ depending only on $N, K$ and $C$.
\end{lemma}
\begin{proof} Let $\ell$ be a $N$--Morse geodesic joining $e$ and $x$. We know by Lemma 2.5 \cite{charney-sultan} that $q(\ell)$ is bounded Hausdorff distance from an $N'$--Morse geodesic $\alpha \colon [0,d] \to Y$ with $\alpha(0)=q(e)$ and  $\alpha(d)=q(x)$.
\end{proof}
Here a quasi--isometry is essential. Every space $X$ quasi--isometrically embeds into $X\times\R^2$ but the latter has no Morse geodesic rays.

\subsection{Boundaries of stable subsets}

Now applying the results in Section \ref{sec:visbdry} we obtain a sequential boundary $\partial  X^{(N)}_e$ for each $X^{(N)}_e$ equipped with a visual metric $d_{(N)}$ (chosen up to quasi--symmetry) with visibility parameter $\varepsilon_{(N)}$. Given $N'\geq N$ the natural inclusion of $X^{(N)}_e$ into $X^{(N')}_e$ defines an injective map $\partial X^{(N)}_e \into \partial X^{(N')}_e$.

Notice that the Gromov product of $x,y\in\partial  X^{(N)}_e$ depends on $N$. For this reason we make the following definition.

\begin{definition} Let $x,y\in\partial  X^{(N)}_e$. We define the $N$\textbf{--Gromov product} of $x$ and $y$, $\Ngp{x}{y}$ to be the supremum of $\liminf_{m,n\to\infty}\set{(x_n\cdot y_m)_e}$ over all pairs of sequences $(x_n),(y_m)\subseteq X^{(N)}_e$ with $\lim x_n=x$ and $\lim y_m=y$.
\end{definition}

These Gromov products differ by at most an additive error.

\begin{lemma}\label{lem:compareGrProd} Let $x, y \in\partial  X^{(N)}_e$ and suppose $N\leq N'$. Then
\[
\Ngp{x}{y} \leq \Npgp{x}{y} \leq \Ngp{x}{y} + 32N'(3,0).
\]
\end{lemma}
\begin{proof} Note that the convergence of a sequence does not depend on $N$, so the first inequality follows because convergent sequences in $X^{(N)}_e$ are contained, and converge, in $X^{(N')}_e$.

By Lemma \hyperref[lem:gromovproduct]{\ref{lem:gromovproduct}(\ref{lem:canpickconvergentseq})} we know that there exist sequences $(x_n), (y_n)$ such that $x=\lim_n x_n$ and $y=\lim_n y_n$ and $\Ngp{x}{y}=\lim_n(x_n \cdot y_n)_e$. 

Since $N \leq N'$ we note that $(x_n)$ and $(y_n)$ are in $X^{(N')}_e$ and by Lemma \hyperref[lem:gromovproduct]{\ref{lem:gromovproduct}(\ref{lem:uptoreplacement})} we know $$\Npgp{x}{y}-32N'(3,0) \leq \liminf(x_n \cdot y_n)=\Ngp{x}{y},$$ where the final equality follows from our choice of $(x_n)$ and $(y_n)$.
\end{proof}

From this we can deduce Theorem \hyperref[bthm:key]{\ref{bthm:key}.\ref{key:boundary}}.

\begin{proposition}\label{prop:quasisym} Let $X$ be a geodesic metric space and let $N,N'$ be Morse gauges such that $N\leq N'$.

The injection $\iota:\partial  X^{(N)}_e\to \partial  X^{(N')}_e$ induced by the inclusion $X^{(N)}_e\subseteq X^{(N')}_e$ is quasi--symmetric onto its image.
\end{proposition}

\begin{proof} Firstly it is clear from Lemma \ref{lem:geodlinspeed} that $\iota$ is injective. The bounds on the metric in Theorem \ref{thm:GrProdmetric} prove that it is a homeomorphism onto its image.

Fix $N,N'$, and choose $\varepsilon,\varepsilon'$ using Theorem \ref{thm:GrProdmetric} such that we obtain metrics $d_\varepsilon$ on $\partial X^{(N)}_e$ and $d_{\varepsilon'}$ on $\partial X^{(N')}_e$.

For the comparison it is sufficient to use $\rho_\varepsilon$ and $\rho_{\varepsilon'}$ via the inequality proved in Theorem \ref{thm:GrProdmetric}.

Fix $x,y,z$ distinct. Now
\[
 \frac{\displaystyle \rho_{\varepsilon'}(\iota(x),\iota(y))}{\displaystyle \rho_{\varepsilon'}(\iota(x),\iota(z))} = \exp(-\varepsilon'(\Npgp{x}{y}-\Npgp{x}{z})),
\]
while $\rho_\varepsilon(x,y)/\rho_\varepsilon(x,z) = \exp(-\varepsilon(\Ngp{x}{y}-\Ngp{x}{z}))$. By Lemma \ref{lem:compareGrProd} we know
\[
 \exp(-\varepsilon'(\Npgp{x}{y}-\Npgp{x}{z})) \leq \exp(-\varepsilon(\Ngp{x}{y}-\Ngp{x}{z}))^{\varepsilon'/\varepsilon}\cdot\exp(32N'(3,0)\varepsilon').
\]
Setting $\eta(r)=r^{\varepsilon'/\varepsilon}\cdot\exp(-32N'(3,0)\varepsilon')$ we see that this injection is quasi--symmetric.
\end{proof}

One immediate consequence of this is that our boundaries generalise hyperbolic boundaries (Theorem \hyperref[bthm:key]{\ref{bthm:key}.\ref{key:Gromov}}).

\begin{theorem}\label{thm:comparehypbdry} Let $X$ be a hyperbolic geodesic metric  space and let $\partial_\infty(X)$ be the usual Gromov boundary of $X$ equipped with a visual metric. For all Morse gauges $N$, there is a map $\partial  X^{(N)}_e\to \partial_\infty(X)$ which is a quasi--symmetry onto its image. As geodesics in a hyperbolic space are uniformly Morse, for any sufficiently large $N$, this map is a quasi--symmetry.
\end{theorem}
\begin{proof}
It follows immediately from Lemma \ref{lem:deltaslim} that a geodesic metric space $X$ is hyperbolic if and only if there exists some $e\in X$ and a constant $C$ such that every $(3,0)$-quasi--geodesic from $e$ to any $x\in X$ is contained in the $C$-neighbourhood of some geodesic from $e$ to $x$.

In particular, if $N(3,0)\geq C$ then the Gromov product and the $N$--Gromov product coincide, so the map $\partial_M^N(X)\to \partial_\infty(X)$ is a quasi--symmetry. The general case then follows from Proposition \ref{prop:quasisym}.
\end{proof}

If $X$ is proper then $\partial^N_M X_e$ is equal as a set to $\partial  X^{(N)}_e$ and the topology defined by the metric is exactly the convergence topology on the Morse boundary presented in \cite{Cordes15}. Note that in \cite{Cordes15} the Morse boundary is only defined for proper geodesic spaces.

 \begin{theorem} \label{thm:seqandraybijection} If $X$ is a proper geodesic metric space then there is a homeomorphism between $\partial  X^{(N)}_e$ and $\partial_M^N X_e$.
\end{theorem}

\begin{proof} Firstly we define a bijection.

For every geodesic ray $\ell \colon \mathbb{R}_{\geq 0} \to X$ with $\ell(0)=e$ we can associate a sequence $(\ell(n))_{n \in \mathbb{N}}$.  It is clear that $(\ell(i) \cdot \ell(j)) \to \infty$ as $i,j \to \infty$ and if we have two equivalent rays $\ell_1$ and $\ell_2$ then $(\ell_1(i) \cdot \ell_2(j))_e \to \infty$ as $i,j \to \infty$.

Let $(x_i)$ be a sequence which converges at infinity. For every $i$ let $\alpha_i$ be an $N$--Morse geodesic from $e$ to $x_i$. Since $X$ is proper we know that there exists a subsequence $(\alpha_{k(i)})$ of $(\alpha_i)$ which converges uniformly to a geodesic ray $\ell$ with basepoint $e$. By \cite[Lemma $2.8$]{Cordes15} $\ell$ is $N$--Morse. Now assume $(x_i)$ and $(y_j)$ are equivalent sequences converging at infinity and let $(\alpha_k)$ and $(\beta_l)$ be the subsequences of geodesics that converge to the $N$--Morse geodesic rays $\ell_1$ and $\ell_2$ respectively. 

Since  $(\alpha_k)$ and $(\beta_l)$ converge uniformly on compact sets, we know that given an $\epsilon> 0$  and $s \in [0, \infty)$ there exists a $K\in \N$ such that the Hausdorff distances between $\ell_1([0,s]$ and $\alpha_k([0,s])$ and between $\ell_2([0,s]$ and $\alpha_l([0,s])$ are less than $\epsilon$ for all $k,l\geq K$. We can choose a $K$ which also satisfies $s< \Ngp{x_K}{y_K}=D$.  Since $X^{(N)}_e$ is $4N(3,0)$--hyperbolic we know that the Hausdorff distance between the geodesic segments $\alpha_k([e, D])$ and $\beta_l([e, D])$ is less than $28N(3,0)$ \cite[Lemma 9.51]{Kapovich-Drutu}. Thus we deduce that for any $s \in [0, \infty)$, the Hausdorff distance between $\ell_1([0,s])$ and $\ell_2([0,s])$ is less than $28N(3,0)+2 \epsilon$. Letting $\epsilon \to 0$ we see that the Hausdorff distance between $\ell_1$ and $\ell_2$ is at most $28N(3,0)$, so $\ell_1$ and $\ell_2$ are equivalent.

To see that this bijection is a homeomorphism it suffices to show that, given a sequence $(\ell_n)$ in $\partial_M^N X$, then $\ell_n \to \ell$ in the sense of the topology on the Morse boundary defined in \cite{Cordes15} if and only if $\Ngp{\ell_n}{\ell} \to \infty$ as $n \to \infty$. 

The proof of this follows from \cite[III.H.3.17]{bh} using Lemma \ref{lem:gromovproduct} for the estimates.
\end{proof}

From this we immediately deduce \hyperref[bthm:key]{\ref{bthm:key}.\ref{key:Mboundary}}.

\subsection{Invariance of the metric Morse boundary}

It only remains to prove  Theorem \hyperref[bthm:key]{\ref{bthm:key}.\ref{key:quasiisom}}. Our method will allow us to prove the stronger result stated in the introduction as Theorem \hyperref[bthm:key]{\ref{bthm:key}.\ref{key:quasiisom}'}, which we prove at the end of the section.

Given a geodesic metric space $X$ and a basepoint $e$ we have associated a collection of (quasi--symmetry classes of) metric spaces $(\partial  X^{(N)}_e,d_{(N)})_N$ indexed by Morse gauges $N$. We call this collection of spaces the \textbf{metric Morse boundary} of $X$ based at $e$.

We now place an equivalence relation on collections of metric spaces $(Z^{(N)},d_{(N)})_N$ indexed by Morse gauges.

We say $(Z^{(N)},d_{(N)})_N$ is \textbf{subsumed} by  $(W^{(N)},d'_{(N)})_N$ if, for every $N$ there exists some $N'$ and an injection $\iota:Z^{(N)}\to W^{(N')}$ which is quasi--symmetric onto its image.

Two collections are \textbf{equivalent} if they subsume each other.

\begin{proposition}\label{prop:baseptinv} Let $X$ be a geodesic metric space and let $e,f\in X$. For every Morse gauge $N$ there exists a Morse gauge $N'$ such that there is an injection $\iota:\partial X^{(N)}_e\to \partial X^{(N')}_f$ which is quasi--symmetric onto its image. In particular, change of basepoint is an equivalence of metric Morse boundaries.
\end{proposition}

\begin{proof} By universality (see Lemma \ref{lem:straightentoNprime}) for every $N$, $X^{(N)}_e$ is a subset of $X^{(N')}_f$ for some $N'$ which depends only on $N$ and $d(e,f)$. Likewise, $X^{(N')}_f\subseteq X^{(N'')}_e$ for some $N''$, so by Lemma \ref{lem:compareGrProd} there exists some constant $D$ such that for all $x,y\in\partial X^{(N)}_e$,
\[
 (x\cdot_{N'} y)_f - D \leq (x\cdot_{N} y)_e \leq (x\cdot_{N'} y)_f + D.
 \]
 The fact that the map $\iota\colon \partial X^{(N)}_e \to \partial X^{(N')}_f$ is quasi--symmetric onto its image now follows immediately from the proof of Proposition \ref{prop:quasisym}.
 \end{proof}

The rest of the section is devoted to the most important theorem. Our goal is to show that a quasi--isometry of geodesic spaces $q \colon X\to Y$ defines a collection of quasi--symmetries (cf. Definition $\ref{defn:quasisym}$) $\partial^{(N)} q \colon \partial X^{(N)} \to \partial Y^{(N')}$.

\begin{theorem} \label{thm:qitoquasisym} Let $q \colon X\to Y$ be a $(K,C)$ quasi--isometry between geodesic metric spaces and let $N$ be a Morse gauge. There is some $N'(K,C,N)$ such that the injection $\partial q_N \colon \partial X^{(N)}_e\to \partial Y^{(N')}_{q(e)}$ is a quasi--symmetry onto its image. In particular, a quasi--isometry of spaces defines an equivalence of metric Morse boundaries.
\end{theorem}

The starting point is to adapt the following proposition to our setting.

\begin{proposition}\label{prop:qidiffofGrPr} \cite[Proposition 5.15(2)]{GhysdlH}
Given $K\geq 1$, $C,\delta\geq 0$, there is a constant $A=A(K,C,\delta)$ such that the following holds.

Let $X,Y$ be two geodesic $\delta$--hyperbolic metric spaces and let $q:X\to Y$ be a $(K,C)$-quasi--isometric embedding. If $w,x,y,z\in X$, then
\[
 K^{-1}\abs{(x\cdot y)_w-(x\cdot z)_w} - A \leq \abs{(q(x)\cdot q(y))_{q(w)}-(q(x)\cdot q(z))_{q(w)}} 
\]
\[
 \leq K\abs{(x\cdot y)_w-(x\cdot z)_w} + A.
\]
\end{proposition}
The original statement includes the hypothesis that $q$ is coarsely onto but this is not used in the proof.

\medskip
We would like to apply this to the restriction of $q$ to each $X^{(N)}_e$. Lemma \ref{lem:straightentoNprime} allows us to restrict the image of the quasi--isometry to a stable subset $Y^{(N')}_{q(e)}$.

This still leaves one obstruction. The spaces $X^{(N)}_e, Y^{(N')}_{q(e)}$ are not geodesic. It easily follows from Lemma \ref{lem:geodlinspeed} that every pair of points in $X^{(N)}_e$ can be connected by a $(1,12N(3,0))$-quasi--geodesic $q\subset X^{(N)}_e$.

However, the proof of \ref{prop:qidiffofGrPr} only relies on the fact that geodesic triangles in $X$ with vertices $w,x,y\in X^{(N)}_e$ are uniformly slim, but this follows directly from Lemma \ref{lem:deltaslim} in the case where $w=e$. Hence we have

\begin{proposition}\label{prop:adaptqidiffofGrPr} 
Given a Morse gauge $N$ and constants $K\geq 1$, $C,\delta\geq 0$, there is a constant $A=A(N,K,C,\delta)$ such that the following holds.

Let $X,Y$ be two geodesic metric spaces and let $q\colon X\to Y$ be a $(K,C)$-quasi--isometry. Let $x,y,z\in X^{(N)}_e$, then there exists some Morse gauge $N'$ such that $q(x),q(y),q(z)\in Y^{(N')}_{q(e)}$ and
\[
 K^{-1}\abs{(x\cdot y)_e-(x\cdot z)_e} - A \leq \abs{(q(x)\cdot q(y))_{q(e)}-(q(x)\cdot q(z))_{q(e)}} 
\]
\[
 \leq K\abs{(x\cdot y)_e-(x\cdot z)_e} + A.
\]
In particular, there exists some $A'$ such that for any $x,y,z\in \partial X^{(N)}_e$ we have
\[
 K^{-1}\abs{\Ngp{x}{y}-\Ngp{x}{z}} - A' \leq \abs{\Npgpq{q(x)}{q(y)}-\Npgpq{q(x)}{q(z)}} 
\]
\[
 \leq K\abs{\Ngp{x}{y}-\Ngp{x}{z}} + A'.
\]
\end{proposition}
\begin{proof} %The first pair of inequalities follow from \ref{prop:qidiffofGrPr} and the discussion above. For the final part, using Lemma \ref{lem:gromovproduct}(\ref{lem:canpickconvergentseq}) we may choose sequences $(x_n),(x'_n),(y_n),(z_n)\subset X^{(N)}_e$ such that $\lim(x_n)=\lim(x'_n)=x$, $\lim(y_n)=y$, $\lim(z_n)=z$, $\lim((x_n\cdot y_n)_e)=(x\cdot y)_e$ and $\lim((x'_n\cdot z_n)_e)=(x\cdot z)_e$. Using the first part we see that
The first pair of inequalities follow from \ref{prop:qidiffofGrPr} and the discussion above. For the final part choose sequences $(x_n),(y_n),(z_n)\subset X^{(N)}_e$ such that $\lim(x_n)=x$, $\lim(y_n)=y$, and $\lim(z_n)=z$. Using the first part we see that
\[
 K^{-1}\abs{(x_n\cdot y_n)_e-(x_n\cdot z_n)_e} - A \leq \abs{(q(x_n)\cdot q(y_n))_{q(e)}-(q(x_n)\cdot q(z_n))_{q(e)}} 
\]
\[
 \leq K\abs{(x_n\cdot y_n)_e-(x_n\cdot z_n)_e} + A.
\]
%Next, observe that by Lemma \ref{lem:gromovproduct}(\ref{lem:uptoreplacement}) $\liminf_{n,m}{(x_n\cdot z_m)_e}\in [(x\cdot z)_e-2\delta, (x\cdot z)_e]$, so the above bounds pass to the limit inferior, with $A$ replaced by $A+2K\delta$.
Next, observe that by Lemma \ref{lem:gromovproduct}(\ref{lem:uptoreplacement})  $$\liminf_{n,m}{(x_n\cdot y_m)_e}\in [(x\cdot y)_e-2\delta, (x\cdot y)_e] \quad \text{and} \quad \liminf_{n,m}{(x_n\cdot z_m)_e}\in [(x\cdot z)_e-2\delta, (x\cdot z)_e],$$ so the above bounds pass to the limit inferior, with $A$ replaced by $A+4K\delta$.
Since $Y_{q(e)}^{(N')}$ is $8N'(3,0)$--hyperbolic,
\[
 \abs{\Npgpq{q(x)}{q(y)} - \Npgpq{q(x)}{q(z)}}-16N'(3,0)
\]
\[
 \leq \liminf\abs{(q(x_n)\cdot q(y_n))_{q(e)}-(q(x_n)\cdot q(z_n))_{q(e)}} 
\leq \abs{\Npgpq{q(x)}{q(y)} - \Npgpq{q(x)}{q(z)}}.
\]
This completes the proof.
\end{proof}

\begin{remark} By Lemma \ref{lem:geodlinspeed} the spaces $X^{(N)}_e$ are $(1,12N(3,0))$-quasi--geodesic, so Proposition \ref{prop:adaptqidiffofGrPr} also follows from \cite[Proposition $5.5$]{BonkSchramm}.
\end{remark}

\begin{proof}[Proof of Theorem $\ref{thm:qitoquasisym}$]
We will work with appropriate pseudo--metrics $\rho_\varepsilon$ on $\partial  X^{(N)}_e$ and $\rho_{\varepsilon'}$ on $\partial  Y^{(N')}_{q(e)}$.

Fix $x,y,z\in \partial  X^{(N)}_e$ distinct. Recall that
$\rho_{\varepsilon}(x,y)=\exp(-\varepsilon\Ngp{x}{y})$, so
\[
 \frac{\displaystyle \rho_\varepsilon(x,y)}{\displaystyle \rho_\varepsilon(x,z)} = \exp(-\varepsilon(\Ngp{x}{y}-\Ngp{x}{z})).
\]
In the first case we assume $\Ngp{x}{y}-\Ngp{x}{z}\geq 0$. Now applying Proposition \ref{prop:adaptqidiffofGrPr} we see that
\[
 \frac{\displaystyle \rho_\varepsilon(x,y)}{\displaystyle \rho_\varepsilon(x,z)} \leq \exp\left(-\varepsilon (K^{-1}(\Npgpq{q(x)}{q(y)}-\Npgpq{q(x)}{q(z)})-A')\right) 
\]
\[
 \leq \exp(A'\varepsilon)\exp\left(-\varepsilon' (\Npgpq{q(x)}{q(y)}-\Npgpq{q(x)}{q(z)})\right)^{K^{-1}\varepsilon/\varepsilon'}.
\]
Hence $\partial^{(N)} q$ is a quasi--symmetry onto its image.

If, however, $\Ngp{x}{y}-\Ngp{x}{z}< 0$ then we use the other bound in Proposition \ref{prop:adaptqidiffofGrPr} and conclude similarly.
\end{proof}
\begin{remark} \label{rmk:itoquasisymm} The only time that a quasi--isometry is required in the the proof of Theorem \ref{thm:qitoquasisym} is in Lemma \ref{lem:straightentoNprime}. The proof of Theorem \ref{thm:qitoquasisym} follows if for each Morse gauge $N$ there exist Morse gauges $N',N''$ and quasi--isometric embeddings $f^{(N)} \colon X_e^{(N)} \to Y_{f(e)}^{(N')}$ and $g_{(N)} \colon Y_p^{(N)} \to X_{g(p)}^{(N'')}$. 
 \end{remark}

\begin{proof}[Proof of Theorem ${\ref{bthm:key}.\ref{key:quasiisom}'}$] One implication follows from the remark above. 

The other implication is an immediate consequence of \cite[Theorems $7.4$ and $8.2$]{BonkSchramm}. In the paper, they assume that the quasi--symmetry is onto, but this is only used in the proof to show that the quasi--isometric embedding they construct is coarsely onto.
\end{proof}

\section{Stable and Morse capacity dimensions}
We assign two notions of dimension to a geodesic metric space.

The first controls the behaviour of stable subsets, the second controls terms in the metric Morse boundary. We first recall the definition of asymptotic dimension.

\begin{definition} A metric space $X$ has  \textbf{asymptotic dimension at most} $n$ (${asdim}(X)\leq n$), if for every $R>0$ there exists a cover of $X$ by uniformly bounded sets such that every metric $R$--ball in $X$ intersects at most $n+1$ elements of the cover. We say $X$ has {\bf asymptotic dimension $n$} if ${asdim}(X) \leq n$ but ${asdim}(X) \nleq n-1$.
\end{definition}

\begin{definition}\label{def:stableasdim} Let $X$ be a geodesic metric space. The \textbf{stable asymptotic dimension} of $X$ ($\stabdim(X)$) is the supremum of the asymptotic dimensions of all stable subsets of $X$.
\end{definition}

By universality (Theorem \hyperref[bthm:key]{\ref{bthm:key}.\ref{key:stable}}) this is equal to the supremum of the asymptotic dimensions of the sets $X^{(N)}_e$.

\begin{definition}\label{def:Morsecapdim} Let $X$ be a geodesic metric space. The \textbf{Morse boundary capacity dimension} of $X$ ($\Mcd(X)$) is defined to be the supremum of the capacity dimensions of the spaces $\partial X^{(N)}_e$ considered with their visual metrics.
\end{definition}

The following four results follow immediately from work in the previous section.

\begin{proposition} Let $X$ be a geodesic metric space. If $N\leq N'$ then ${cdim}(\partial  X^{(N)}_e)\leq {cdim}(\partial  X^{(N')}_e)$.
\end{proposition}
\begin{proof} By Proposition \ref{prop:quasisym} there is a map $\iota(N,N'):\partial  X^{(N)}_e\to\partial  X^{(N')}_e$ which is a quasi--symmetry onto its image.

Any quasi--symmetry of $\partial  X^{(N')}_e$ induces a quasi--symmetry on the image of $\iota(N,N')$, so $\mathrm{cdim}(\partial  X^{(N')}_e)\geq \mathrm{cdim}(\textup{im}\iota)=\mathrm{cdim}(\partial  X^{(N)}_e)$.
\end{proof}

\begin{theorem}\label{thm:cdimbpinv} The Morse boundary capacity dimension is basepoint independent. 
\end{theorem}

\begin{theorem}\label{thm:cdimqiinv} The stable dimension and Morse boundary capacity dimension are quasi--isometry invariants of geodesic metric spaces.
\end{theorem}

\begin{theorem}\label{thm:capdimhyp} Let $X$ be a hyperbolic geodesic metric space. Then $\stabdim(X)$ is equal to the asymptotic dimension of $X$ and $\Mcd(X)$ is equal to the capacity dimension of the Gromov boundary of $X$.
\end{theorem}

\begin{theorem}\label{thm:hypembsubsets}
Let $X$ be a geodesic metric space. Then $\Mcd(X)$ is equal to the supremum of the capacity dimensions of the sequential boundaries of all stable subsets $Y$ of $X$.
\end{theorem}
\begin{proof}
One inequality is clear as each $X^{(N)}_e$ is stable by Lemma \ref{lem:deltaslim}.

Let $Y\subseteq X$ be stable and let $y\in Y$. By assumption $Y\subseteq X^{(N)}_y$ and by Lemma \hyperref[lem:gromovproduct]{\ref{lem:gromovproduct}(\ref{lem:uptoreplacement})} the Gromov product of two points $x,y\in\partial  Y$ in $Y$ differs from that in $X^{(N)}_y$ by a uniform constant. Applying the proof of Proposition \ref{prop:quasisym} $\partial  Y$ quasi--symmetrically embeds in $\partial  X^{(N)}_y$. 

Therefore $\mathrm{cdim}(\partial  Y)\leq \mathrm{cdim}(\partial  X^{(N)}_y)\leq \Mcd(X)$.
\end{proof}

\begin{remark} Since the conformal dimension (introduced by Pansu in \cite{PansuConf89}) is also a quasi--symmetry invariant, one could also define a conformal dimension of the Morse boundary in the same manner and derive the properties listed above.
\end{remark}

We now look at methods of controlling the stable asymptotic dimension and the Morse boundary capacity dimension. The first bound is an immediate application of \cite[Proposition $3.6$]{MackSis}.

\begin{proposition}\label{prop:upbdMorsedim}
Let $X$ be a geodesic metric space. Then $\Mcd(X)\leq\stabdim(X)$.
\end{proposition}

\begin{proposition}\label{prop:lwbdMorsedim} Let $X$ be a geodesic metric space. Then $\Mcd(X)+1\geq \stabdim(X)$.
\end{proposition}
\begin{proof} Each $\partial X^{(N)}_e$ is visual by construction, so by \cite[Theorem $1.1$]{Buycapdim}, $\mathrm{cdim}(\partial X^{(N)}_e) + 1\geq \mathrm{asdim}(X^{(N)}_e)$.
\end{proof}

Corollary \ref{bcor:comparedims} follows immediately from the above two propositions.

\section{Right--angled Artin groups} \label{sec:RAAG}

The goal of this section is to prove Theorem \ref{thm:RAAG}.

Let $\Gamma$ be a finite, simplicial graph with vertex set $V$. The {\bf right--angled Artin group} associated to $\Gamma$ is the group $A_\Gamma$ with presentation $$A_\Gamma = \fpres{V}{\setcon{v_iv_jv_i^{-1}v_j^{-1}}{v_i, v_j \in V\text{ are connected by an edge in } \Gamma}}.$$

One useful cube complex associated to a right--angled Artin group $A_\Gamma$ is the {\bf Salvetti complex} $S_\Gamma$. To form the Salvetti complex, start with a wedge of $\abs{V}=n$ oriented circles and label the edges with the verticies. Call the basepoint $x_0$. For each edge in $\Gamma$, take a square and glue the edges of the square along the commutation relation of the two generators corresponding to the endpoints of the edge, i.e., glue in a $2$--torus. For each triangle in $\Gamma$, attach a $3$--torus with faces that correspond to the tori of the three edges. We continue by attaching a $k$--torus for each set of $k$ mutually commuting generators in the same way. (Equivalently, we glue in a $k$--torus for each complete graph with $k$ vertices in $\Gamma$.) It is easy to verify that $S_\Gamma$ has fundamental group $A_\Gamma$ and the link of $x_0$ is flag. Thus, the universal cover, $\tilde{S}_\Gamma$, is a $\mathrm{CAT}(0)$ cube complex. 

A {\bf hyperplane} or {\bf wall} in a $\mathrm{CAT}(0)$ cube complex is an equivalence class of midplanes of cubes where the equivalence relation is generated by the rule that two midplanes are related if they share a face. We say two hyperplanes $H_1$, $H_2$ in a $\mathrm{CAT}(0)$ cube complex are {\bf strongly separated} if $H_1 \cap H_2 = \emptyset$ and no hyperplane intersects both $H_1$ and $H_2$. 

If $\Gamma_1$ and $\Gamma_2$ are two graphs, their {\bf join} is the graph obtained by connecting every vertex of $\Gamma_1$ to every vertex of $\Gamma_2$ by an edge. A {\bf join subgroup} of a right--angled Artin group $A_\Gamma$ is subgroup induced by a join subgraph of $\Gamma$.

Another useful tool we will use is the contact graph, $\mathcal{CG}$, which was defined by Hagen in \cite{hagen:2014aa}. The vertices of the contact graph are the set $\mathcal{W}$ of hyperplanes of a $\mathrm{CAT}(0)$ cube complex. We connect two hyperplanes if they have the {\bf contact relation}: distinct hyperplanes $H_1, H_2 \in \mathcal{W}$ contact if they have dual 1-cubes $c_1, c_2$ that have a common $0$-cube or if $H_1$ and $H_2$ cross. In the same paper, Hagen proves that $\mathcal{CG}$ is quasi--isometric to an $\mathbb{R}$-tree.

Consider the $1$--skeleton of $\tilde{S}_\Gamma$, $\tilde{S}_\Gamma^{(1)}$, with the path metric $d_{(1)}$. If we choose a vertex $e$ in $\tilde{S}_\Gamma^{(1)}$ as a basepoint, then $\tilde{S}_\Gamma^{(1)}$ is a Cayley graph of $A_\Gamma$ where the vertices are labeled by elements of $A_\Gamma$ and edges by elements in the standard generating set of $A_\Gamma$ (i.e., the vertices of $\Gamma$). For a generator $v$, let $e_v$ denote the edge from the basepoint $e$ to the vertex $v$. Any edge in $\tilde{S}_\Gamma$ determines a unique wall that is dual to that edge (the wall containing the midpoint of the edge). We denote the wall dual to $e_v$ by $H_v$. 

For a cube in  $\tilde{S}_\Gamma$, all of the parallel edges are labelled by the same generator $v$. Thus it follows that all of the edges crossing a wall $H$ have the same label $v$, we call this a wall of {\bf type $v$}. Since $A_\Gamma$ acts transitively on edges labeled $v$, a wall is of type $v$ if and only if it is a translate of the standard wall $H_v$. 

\begin{theorem}\label{thm:RAAGtoqtree} Let $A_\Gamma$ be a right--angled Artin group. If $\Gamma$ is a single vertex then $A_\Gamma$ is stably equivalent to a line, if $\Gamma$ is a complete graph with at least $2$ vertices then $A_\Gamma$ is stably equivalent to a point. In all other cases $A_\Gamma$ is stably equivalent to a regular $3$-valent tree.

As a result, $\stabdim(A_\Gamma)\leq 1$, with equality unless $A_\Gamma$ is abelian of rank at least $2$.
\end{theorem}
If $A_\Gamma$ is free or abelian, then we are done. Assume $A_\Gamma$ is not free and not abelian.

We will use the $d_{(1)}$ metric instead of the $\mathrm{CAT}(0)$ metric, i.e., all paths will be edge paths in $\tilde{S}_\Gamma^{(1)}$. Since the $\mathrm{CAT}(0)$ metric is bounded above and below by linear functions of the $d_{(1)}$ metric, depending only on the number of generators in $A_\Gamma$, the result will follow.

Let $x \in \tilde{S}_\Gamma^{(N)}$. By definition that means there exits an $N$--Morse geodesic $\ell \colon [a,b] \to \tilde{S}_\Gamma^{(1)}$  such that $\ell(a)=e$ and $\ell(b)=x$. Associated to $\ell$ is a sequence $\{H_1, H_2, \ldots, H_m\}$ of hyperplanes dual to each edge in $\ell$. We define a map $q \colon \tilde{S}_\Gamma^{(N)} \to \mathcal{CG}$ by $q(x)=H_m$. This map is well defined up to one edge in $\mathcal{CG}$, so we are left to show that this a quasi--isometric embedding.

We first show that if $\ell \colon [a,b] \to \tilde{S}_\Gamma$ is an $N$--Morse geodesic, then $\ell$ quasi--isometrically embeds into $\mathcal{CG}$. Let $x, y \in \ell$. From the definition of the contact graph we see that $d_{ \mathcal{CG}}(q(x),q(y)) \leq d(x,y)$. %, since $H_i$ and $H_{i+1}$ are hyperplanes dual to one cubes that have a common $0$-cube. 

To prove an inequality in the other direction we appeal to the following claim about Morse geodesics in $\tilde{S}_\Gamma$: Let $\ell$ be an $N$--Morse geodesic segment. Then there exists a $r \geq 0$ (depending only on $N$) such that $\ell$ crosses a sequence of strongly separated hyperplanes $\{F_i\}_{i=1}^s$ such that $d(F_i \cap \ell, F_{i+1}\cap \ell)<r$ for all $i$. We note this claim is a stronger version of the more general claim in Theorem 4.2 in \cite{charney-sultan}.

Proof of claim: Fix a minimal word for $\ell$ and let $\mathcal{H}$ be the sequence of walls crossed by the corresponding edge path. We build a subsequence, $\mathcal{F} \subset \mathcal{H}$, inductively starting with the hyperplane dual to the first edge of $\ell$, $F_1=H_{v_1}$, and given $F_k=g_iH_{v_{i+1}},$ we choose $F_{k+1}=g_jH_{v_{j+1}}$ to be the next wall in $\mathcal{H}$ strongly separated from $F_k$. Following the proof of Lemma 3.3 in \cite{behrstock-charney} we see that  ${g_i}^{-1}g_j$ lies in the product of three join subgroups. By construction, consecutive hyperplanes in $\mathcal{F}$ are strongly separated. By Theorem 2.5 of \cite{behrstock-charney} we know that $\mathcal{F}$ is in fact a sequence of pairwise strongly separated hyperplanes. We also know that since $\ell$ is $N$--Morse that the distance $\ell$ spends in a join subgroup is bounded above by a constant $r'>0$ depending only $N$ (because join subgroups create geodesically convex product spaces in $\tilde{S}_\Gamma$). We set $r=3r'$ and the claim follows.
\medskip

With the claim proved we now prove the other inequality. Let $x,y \in \ell$ and let $\{F_i\}_{i=1}^k$ be the hyperplanes in $\mathcal{F}$ between $x$ and $y$. We first observe that  $\frac{d(x,y)-2r}{r} < k$. Next observe that no two non-consecutive hyperplanes in $\{F_i\}$ can have the contact relation, and if two non-consecutive hyperplanes $F_i, F_j$ share a contact relation with a third hyperplane $H$ then $F_t$ crosses $H$ for $i<t<j$. But, if $j>i+1$ it would violate the fact that the sequence $\{F_i\}_{i=1}^s$ are strongly separating hyperplanes. Thus $d_{\mathcal{CG}}(q(x),q(y)) \geq \frac{k}{2}$. We put the two inequalities together to get $\frac{1}{2r}\cdot d(x,y)-1 \leq d_{\mathcal{CG}}(q(x),q(y))$. 

We know by Lemma \ref{lem:deltaslim} that given two points $x,y \in \tilde{S}_\Gamma^{(N)}$, any geodesic $\ell'$ from $x$ to $y$ is $N'$--Morse where $N'$ depends only on $N$. Thus we can follow the same argument as above and deduce that  $\frac{1}{2r'}\cdot d(x,y)-1 \leq d_{\mathcal{CG}}(q(x),q(y))\leq d(x,y)$ for some $r'$ depending on $N'$. Setting $K=2\max\{r, r'\}$ we conclude that $q \colon \tilde{S}_\Gamma^{(N)} \to \mathcal{CG}$ is a $(K, 1)$-quasi--isometric embedding.

It follows that the image of $\tilde{S}_\Gamma^{(N)}$ under $q$ is quasi--convex and the number of points contained in the intersection of this image with a ball of radius $R$ is bounded from above by a function which is at most exponential in $R$. It follows that the convex hull of $\tilde{S}_\Gamma^{(N)}$ is quasi--isometric to a proper simplicial tree. Thus $\tilde{S}_\Gamma^{(N)}$ is quasi--isometric to the same simplicial tree.

\begin{corollary}\label{cor:RAAGCdim0} Let $A_\Gamma$ be a non-abelian right--angled Artin group. For each Morse gauge $N$, the boundary $\partial (A_\Gamma)^{(N)}_e$ is quasi--symmetric to a Cantor set and possibly a finite number of isolated points. Consequently, $\Mcd(A(\Gamma))=0$.
\end{corollary}

\begin{proof} By the preceding theorem, we know $A_\Gamma^{(N)}$ is quasi--isometric to a proper tree and thus $\partial (A_\Gamma)^{(N)}_e$ must be a Cantor space, $\mathcal{C}$, possibly with a finite number of isolated points $\{x_i\}_{i=1}^n$ added. Let $r'=\min_{x_i}\{\inf_{y \in \mathcal{C}}\{d(x_i, y)\} \}$. This number will be positive, since there are only finitely many $x_i$. Since Cantor spaces are locally self similar, their capacity dimension is equal to their topological dimension (which is 0) \cite{BuyLeb}. So, we know there exists some $\delta \in (0,1)$ such that for every sufficiently small $r >0$ there is an open cover $\mathcal{U}$ of $\mathcal{C}$ of sets with diameter less than $r$ with $m(\mathcal{U})=0$ and $L(\mathcal{U}) \geq \delta r$. Without loss we may assume $r < r'$, so we may cover the $x_i$ by balls of radius $r/2$ and add those balls to $\mathcal{U}$.  By definition of $r'$ we know our new cover has multiplicity $0$, so $\Mcd(A_\Gamma^{(N)})=0$. The result follows. 
\end{proof}

\begin{corollary} Any stable subgroup of a right--angled Artin group is virtually free. \end{corollary}

\begin{proof}  Since each $A^{(N)}_\Gamma$ is quasi--isometric to a proper tree, we know by universality of stable subsets (Theorem \hyperref[bthm:key]{\ref{bthm:key}.\ref{key:universal}}) that any stable subgroup of a right--angled Artin group is quasi--isometric to a proper tree. Groups which are quasi--isometric to proper trees are virtually free \cite[Corollary 7.19]{GhysdlH}.
\end{proof}

\begin{proposition} Let $q:\tilde{S}_\Gamma \to \mathcal{CG}$ be the map defined in Theorem \ref{thm:RAAGtoqtree}. Suppose $Y$ is a quasi--convex subset of $\tilde{S}_\Gamma$ and $q|_Y$ is a $(K,C)$-quasi--isometric embedding. Then $Y$ is $N$--stable where $N$ depends on $K, C$. 
\end{proposition}

\begin{proof} Let $Y$ be a quasi--convex subset of $\tilde{S}_\Gamma$ and $q|_Y$ is a $(K,C)$-quasi--isometric embedding. Since $Y$ is quasi-convex, at the cost of increasing $C$, we can assume that all geodesics in $\tilde{S}_\Gamma$ between points in $x,y \in  Y$ are also also $(K,C)$-quasi--isometrically embedded in $\mathcal{CG}$. Let $x,y \in  Y$ and let $\mathcal{H}$ be the sequence of hyperplanes associated to the geodesic segment $[x,y]$ by $q$.  We know that if $a,b \in Y$ and  $d(a,b)=K(C+3)$ then $d_{\mathcal{CG}}(q(a),q(b))\geq 3$. Two hyperplanes which are at least distance 3 in the contact graph strongly separated. Thus we can choose a subsequence of $\mathcal{F} \subset \mathcal{H}$ such that consecutive hyperplanes $F_i, F_{i+1} \in \mathcal{F}$ are strongly separated and $d(F_i \cap [x,y], F_{i+1}\cap [x,y])< K(C+3)+1$. By Theorem 4.2 in \cite{charney-sultan} $[x,y]$ is $N$--Morse where $N$ depends only on $K(C+3)+1$. Thus $Y$ is stable.
\end{proof}

With this proposition we have a classification of quasi--convex subspaces of the universal cover of the Salvetti complex, $\tilde{S}_\Gamma$.

\begin{corollary} A quasi--convex subspace $Y$ of $\tilde{S}_\Gamma$ is stable if and only if  $q|_Y$ is a quasi--isometric embedding. \end{corollary}

\section{Teichm\"uller Space and Mapping Class Groups} \label{sec:Mod}

The goal of this section is to prove Theorem \ref{thm:Mod}.

Let $\Sigma$ be a surface of finite type of genus $g \geq 0$ and $p \geq 0$ punctures such that $3g-3+p \geq 1$. The {\bf curve graph}, $\mathcal{C}(\Sigma)$, is a graph whose vertices are isotopy classes of essential simple closed curves on $S$ and two curves are joined by an edge if they can be realized disjointly on $S$. A celebrated theorem of Masur and Minsky shows that $\mathcal{C}(\Sigma)$ is a $\delta$--hyperbolic space \cite{Masur:1999aa}.

A \emph{complete clean marking} on $\Sigma$ is a set $\mu=\{(\alpha_1, \beta_1), \ldots, (\alpha_m, \beta_m)\}$ where $\{\alpha_1, \ldots, \alpha_m\}$ is a pants decomposition of $\Sigma$ (i.e., a simplex in $\mathcal{C}(\Sigma)$), and each $\beta_i$ is (an isotopy class of) a simple closed curve disjoint from $\alpha_j$ for $i \neq j$ and intersects $\alpha_i$ once (twice) if the surface filled by $\alpha_i$ and $\beta_i$ is a once-punctured torus (four-times-punctured sphere). We call $\set{\alpha_i}$ the {\bf base} of $\mu$ and for every $i$, $\beta_i$ is called the {\bf transverse curve to $\alpha_i$ in $\mu$}.  The marking graph, $\mathcal{M}(\Sigma)$, is a graph whose vertices are (complete clean) markings and two markings $\mu_1, \mu_2 \in \mathrm{V}(\mathcal{M}(\Sigma))$ are joined by an edge if they differ by an elementary move. There is a coarsely well defined map $proj_S \colon \mathcal{M}(\Sigma) \to \mathcal{C}(\Sigma)$, which is defined by mapping the base of a marking $\mu$ to the (1-skeleton of the) simplex it defines in $\mathcal{C}(\Sigma)$. For more information see \cite{Masur:2000aa}.

%Recall, for each $x \in \teich$ there is a {\bf short marking}, which is constructed inductively by picking the shortest curves for the base and repeating for the transverse curves. This defines a map $\Phi \colon \teich \to \mathcal{M}(\Sigma)$. 

\begin{theorem} \label{thm:ModTeichstabequiv} Let $S$ be a surface of finite type. Then $\Mod$ and $\teich$ are stably equivalent. In particular, $\stabdim(\Mod)=\stabdim(\teich)$.
\end{theorem} 

\begin{proof} First note that by a theorem of Masur and Minsky \cite{Masur:2000aa} $\Mod$ is quasi--isometric to $\mathcal{M}(\Sigma)$ and therefore they are stably equivalent, so we will use $\mathcal{M}(\Sigma)$ for the remainder of the proof. Recall that by Lemma \ref{lem:thirdedgeMorse}, any geodesic between points in $\mathcal{M}(\Sigma)^{(N)}$ is $N'$--Morse where $N'$ depends only on $N$. By Lemma 4.9 in \cite{Cordes15} we know that $\mathcal{M}(\Sigma)^{(N)}$ is quasi--isometrically embedded in $\mathcal{T}(\Sigma)^{(N'')}$ where $N''$ and the quasi--isometry constants depend only on $N$. Using the same logic and Lemma 4.10 in \cite{Cordes15}, we conclude that $\mathcal{T}(\Sigma)^{(N)}$ quasi--isometrically embeds into $\mathcal{M}(\Sigma)^{(N''')}$ for some $N'''$. We now invoke Remark \ref{rmk:itoquasisymm} to finish the proof.
\end{proof}

\begin{theorem} Let $\Sigma$ be a surface of finite type of genus $g$ with $p$ punctures. Then 
\begin{align*} \Mcd(\Mod)=\Mcd(\teich) &\leq 4g+p-4 \text{ if } p>0 \\
	 &\leq 4g-5 \text{ if } p=0.
 \end{align*}
\end{theorem}

\begin{proof} Again we will use $\mathcal{M}(\Sigma)$ instead of $\Mod$. Recall that by Lemma \ref{lem:deltaslim} that any geodesic between points in $\mathcal{M}(\Sigma)^{(N)}$ is $N'$--Morse where $N'$ depends on $N$. As above, Lemma 4.9 in \cite{Cordes15} gives us $\mathcal{M}(\Sigma)^{(N)}$ is quasi--isometrically embedded in $\mathcal{C}(\Sigma)$. Thus by Theorem \ref{thm:qitoquasisym} we conclude that $\cdim{\partial  \mathcal{M}(\Sigma)^{(N)}} \leq \cdim{\partial  \mathcal{C}(\Sigma)}$. Since this is true for all Morse gauges $N$, we conclude that $\Mcd(\mathcal{M}(\Sigma)) \leq \cdim{\partial  \mathcal{C}(\Sigma)}$. Given a surface $\Sigma$ of genus $g$ with $p$ punctures, Bestvina--Bromberg show that $\cdim{\partial  \mathcal{C}(\Sigma)} \leq 4g+p-4$ if $p>0$ and $\cdim{\partial  \mathcal{C}(\Sigma)} \leq 4g-5$ if $p=0$ \cite{Bestvina-Bromberg}. The conclusion follows.

The $\teich$ case follows from \ref{thm:ModTeichstabequiv}.
\end{proof}

\begin{corollary} For every $n$ there is a surface $\Sigma$ of finite type such that $\Mod$ admits a stable subset quasi--isometric to $\mathbb{H}^n$.
\end{corollary}
\begin{proof}
This follows from \cite{leininger:2014aa}, \cite[Proposition $4.3$]{Cordes15} and Theorem $\ref{thm:ModTeichstabequiv}$.
\end{proof}

\section{Small cancellation groups}\label{sec:smallcanc}

The goal of this section is to prove Theorems \ref{bthm:smallcanc} and \ref{bthm:Gromovmonster}

We will need a few preliminaries on small cancellation groups.

\begin{definition}\label{defn:piece}
Let $\Gamma$ be a directed edge--labelled graph. A \textbf{piece} is a labelled path $p$ admitting two distinct label--preserving graph homomorphisms $p\to\Gamma$.

Let $\lambda>0$. The graph $\Gamma$ satisfies the $C'(\lambda)$\textbf{ graphical small cancellation condition} if any piece $p$ contained in a simple closed path $C$ in $\Gamma$ satisfies $\abs{p}<\lambda\abs{C}$.
\end{definition}

\begin{definition}\label{defn:diagram} A \textbf{diagram} $D$ is a finite, simply connected 2-complex with a fixed embedding into the plane such that the image of every $1$--cell, called an \textbf{edge}, has an orientation and a label from a fixed set $S$. A \textbf{disc diagram} is a diagram homeomorphic to the $2$--disc. The \textbf{boundary} of a disc diagram $D$, denoted $\partial D$, is its topological boundary inside $\R^2$.

A \textbf{disc component} of a diagram $D$ is a subdiagram homeomorphic to the $2$--disc whose boundary is contained in $\partial D$.
\end{definition}

\begin{definition}\label{defn:Gammadiagram}
Let $\Gamma$ be a finite directed graph whose edges are labelled by a finite set $S$. A $\Gamma$--diagram is a diagram with directed edges labelled by elements of $S$ in which the boundary of every $2$-cell, called a \textbf{face} is isomorphic (as a directed graph with labelled edges) to a simple cycle in $\Gamma$.

A $\Gamma$--diagram is \textbf{reduced} if no vertex is the end vertex of two different edges with the same direction and label.
\end{definition}

Any simple closed path $C$ in a reduced $\Gamma$--diagram defines an element (called the \textbf{label} of $C$) of the free group $F(S\sqcup S^{-1})$ by choosing a starting vertex on $C$ and an orientation and writing the generator $s$ when an edge labelled by $s$ is traversed with its direction and $s^{-1}$ if it is traversed against its orientation. Note that the label of $C$ is unique up to inversion and a choice of cyclically reduced conjugate.

We define the group $G(\Gamma)$ to be the quotient of $F(S\sqcup S^{-1})$ by the normal subgroup generated by all words which can be obtained as the labels of simple closed paths in $\Gamma$.

A subpath $P$ of a diagram $D$ is said to be an \textbf{interior arc} if every interior vertex of $P$ has degree $2$ and $P$ intersects $\partial D$ in at most $2$ points (which, if such intersections exist, are necessarily endpoints of $P$).

\begin{lemma}[{\cite[Lemma~2.13]{Gruber-graphicalC(6)C(7)}}]\label{lem:graphical_diagrams}
Let $\Gamma$ be a $C'(\frac{1}{6})$--labelled graph, and let $w$ be a word in the free monoid generated by $S\sqcup S^{-1}$ which represents the identity in $G(\Gamma)$. Then, there exists an $\Gamma$--diagram with boundary word $w$ in which every interior arc is a piece.
\end{lemma}

\begin{theorem}[Strebel's classification \cite{Str90}]\label{thm:Strebel} Let $\Gamma$ be a directed graph whose edges are labelled by a fixed set $S$ which satisfies the $C'(1/6)$ graphical small cancellation condition. Let $X$ be the Cayley graph of $G(\Gamma)$ with respect to the generating set $S\sqcup S^{-1}$.

If $B$ is a geodesic bigon in $X$ then there is a reduced $\Gamma$--diagram $D_B$ with boundary $B$ and any such diagram satisfies the following: every disc component of $D_B$ is either a single face or of type $I_1$ below.

If $T$ is a geodesic triangle in $X$ then there is a reduced $\Gamma$--diagram $D_T$ with boundary $T$ and any such diagram satisfies the following: every disc component of $D_T$ is either a single face or of one of the types in the figure below. Moreover, at most one of these disc components is not of type $I_1$.
\end{theorem}

\begin{figure*}[h!]
\begin{center}
\begin{tikzpicture}[yscale=1,xscale=1, 
vertex/.style={draw,fill,circle,inner sep=0.3mm}]

\draw[thin] (0,3) .. controls (-0.5,2.33) and (-0.5,1.66) .. (0,1) .. controls (0.5,1.66) and (0.5,2.33) .. (0,3);

\path (0,1) node[below] {I${}_1$};

\begin{scope}
\clip (0,3) .. controls (-0.5,2.33) and (-0.5,1.66) .. (0,1) .. controls (0.5,1.66) and (0.5,2.33) .. (0,3);

\draw[very thin]
				(-2,2.75) -- (2,2.75)  
				(-2,2.5) -- (2,2.5)
				(-2,1.5) -- (2,1.5)
				(-2,1.25) -- (2,1.25);
\end{scope}

\draw[xshift=4.5cm, thin] (-0.3,3) .. controls (-0.5,2.33) and (-0.5,1.66) .. (0,1) .. controls (0.5,1.66) and (0.5,2.33) .. (0.3,3) -- (-0.3,3);

\path (4.5,1) node[below] {I${}_2$};

\begin{scope}[xshift=4.5cm]
\clip (-0.3,3) .. controls (-0.5,2.33) and (-0.5,1.66) .. (0,1) .. controls (0.5,1.66) and (0.5,2.33) .. (0.3,3) -- (-0.3,3);

\draw[very thin]
				(-2,2.75) -- (2,2.75)  
				(-2,2.5) -- (2,2.5)
				(-2,1.5) -- (2,1.5)
				(-2,1.25) -- (2,1.25);

\end{scope}

\draw[xshift=9cm, thin]    (0,3) -- (-1,1) -- (1,1) -- (0,3);

\path (9,1) node[below] {I${}_3$};

\begin{scope}[xshift=9cm]
\clip (0,3) -- (-1,1) -- (1,1) -- (0,3);

\draw[very thin]    (-1,1.9) -- (0,0.9)
					(-1,1.4) -- (0,0.4)
					(-1,1.2) -- (0,0.2)
					(1,1.9) -- (0,0.9)
					(1,1.4) -- (0,0.4)
					(1,1.2) -- (0,0.2);

\end{scope}

\draw[thin]    (0,0) -- (-1,-2) -- (1,-2) -- (0,0);

\path (0,-2) node[below] {II};

\begin{scope}
\clip (0,0) -- (-1,-2) -- (1,-2) -- (0,0);

\draw[very thin]    (-1,-1.2) -- (0,-2.2)
					(-1,-1.6) -- (0,-2.6)
					(-1,-1.8) -- (0,-2.8)
					(1,-1.2) -- (0,-2.2)
					(1,-1.6) -- (0,-2.6)
					(1,-1.8) -- (0,-2.8)
					(-1,-0.8) -- (1,-0.8)
					(-1,-0.5) -- (1,-0.5)
					(-1,-0.3) -- (1,-0.3);
					
\end{scope}

\draw[xshift=3cm, thin]    (0,0) -- (-1,-2) -- (1,-2) -- (0,0);

\path (3,-2) node[below] {III};

\begin{scope}[xshift=3cm]
\clip (0,0) -- (-1,-2) -- (1,-2) -- (0,0);

\draw[very thin]    (-1,-1) -- (0,-2)
					(-1,-1.6) -- (0,-2.6)
					(-1,-1.8) -- (0,-2.8)
					(1,-1) -- (0,-2)
					(1,-1.6) -- (0,-2.6)
					(1,-1.8) -- (0,-2.8)
					(-1,-0.8) -- (1,-0.8)
					(-1,-0.5) -- (1,-0.5)
					(-1,-0.3) -- (1,-0.3);
\end{scope}

\draw[xshift=6cm, thin]    (0,0) -- (-1,-2) -- (1,-2) -- (0,0);

\path (6,-2) node[below] {IV};

\begin{scope}[xshift=6cm]
\clip (0,0) -- (-1,-2) -- (1,-2) -- (0,0);

\draw[very thin]    (-1,-1.2) -- (0,-2.2)
					(-1,-1.6) -- (0,-2.6)
					(-1,-1.8) -- (0,-2.8)
					(1,-1.2) -- (0,-2.2)
					(1,-1.6) -- (0,-2.6)
					(1,-1.8) -- (0,-2.8)
					(-1,-0.8) -- (1,-0.8)
					(-1,-0.5) -- (1,-0.5)
					(-1,-0.3) -- (1,-0.3);
					
	\begin{scope}
		\clip (-1,-1.2) -- (0,-2.2) -- (1,-1.2) -- (1,-0.8) -- (-1,-0.8) -- (-1,-1.2);
		
		\draw[very thin]	(0,-1.4) -- (-1,-2)
							(0,-1.4) -- (1,-2)
							(0,-1.4) -- (0,0);
	\end{scope}
					
\end{scope}

\draw[xshift=9cm, thin]    (0,0) -- (-1,-2) -- (1,-2) -- (0,0);

\path (9,-2) node[below] {V};

\begin{scope}[xshift=9cm]
\clip (0,0) -- (-1,-2) -- (1,-2) -- (0,0);

\draw[very thin]    (-1,-1.2) -- (0,-2.2)
					(-1,-1.6) -- (0,-2.6)
					(-1,-1.8) -- (0,-2.8)
					(1,-1.2) -- (0,-2.2)
					(1,-1.6) -- (0,-2.6)
					(1,-1.8) -- (0,-2.8)
					(-1,-0.8) -- (1,-0.8)
					(-1,-0.5) -- (1,-0.5)
					(-1,-0.3) -- (1,-0.3)
					
					(0,-1.6) -- (0,-2)
					(0,-1.6) -- (-1.6,0)
					(0,-1.6) -- (1.6,0);
\end{scope}

\end{tikzpicture}
\caption{Minimal disc diagrams whose boundary is a locally geodesic triangle}
\end{center}
\end{figure*}
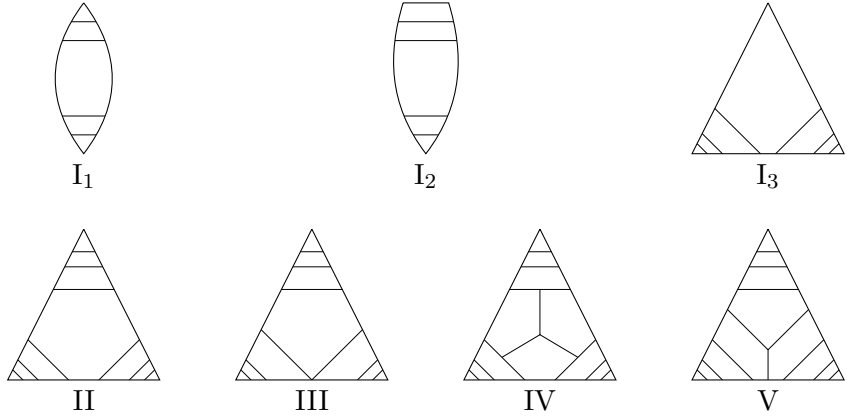
In particular, in the above diagrams, the intersection of the boundary of a face $\Pi$ with a geodesic has length at most $\abs{\partial\Pi}/2$ and the intersection of the boundaries of any two distinct faces $\Pi,\Pi'$ has length strictly less than $\abs{\partial \Pi}/6$.
\medskip

With these basics established we now focus on proving Theorems \ref{bthm:smallcanc} and \ref{bthm:Gromovmonster}.

\begin{theorem}\label{thm:fpresscancasdim2} Let $X$ be the Cayley graph of a $C'(\frac{1}{6})$ graphical small cancellation presentation $G=\fpres{S}{\Gamma}$ where $\Gamma$ is finite. Then $\stabdim(X)-1=\Mcd(X)\leq 1$.
\end{theorem}
\begin{proof} By \cite[Theorem $2.16$]{Gruber-graphicalC(6)C(7)} the group $G$ has cohomological dimension at most $2$, and is hyperbolic by \cite{Ollivier}. The result then follows by \cite{BestMess}.
\end{proof}

The \textbf{girth} of a graph $\Gamma$ is the length of its shortest simple cycle of positive length; we denote it by $g(\Gamma)$.

\begin{theorem}\label{thm:scancstabletofinpres} Let $X$ be the Cayley graph of a $C'(\frac{1}{6})$ graphical small cancellation presentation $\fpres{S}{\Gamma}$ where $\Gamma = \bigsqcup \Gamma_i$, each of the graphs $\Gamma_i$ is finite and there exists some $\delta>0$ such that $g(\Gamma_i)\geq \delta\textrm{diam}(\Gamma_i)$.

For each Morse gauge $N$ there is a finite subgraph $\Gamma_N\subset \Gamma$ and an isometric embedding $X^{(N)}_e\into X_N$ where $X_N$ is the Cayley graph of the presentation $\fpres{S}{\Gamma_N}$.
\end{theorem}

Without loss of generality we assume that $(\abs{\Gamma_i})_i$ is a non--decreasing sequence. It follows that $g(\Gamma_i)\to\infty$ as $i\to\infty$.

Before proceeding with the proof we present one lemma.

\begin{lemma}\label{lem:highgirthnotMorse} Let $\Gamma$ be a finite graph and suppose $g(\Gamma)\geq \delta\textrm{diam}(\Gamma)$ for some $\delta>0$.

There exist constants $\lambda,c,\epsilon$ depending only on $\delta$ such that: for any simple cycle $C$ in $\Gamma$ which is a concatenation of at most six geodesics $\gamma_1,\gamma_2,\dots,\gamma_l$ with $\abs{\gamma_1}\geq \frac{1}{6}\abs{C}$ there is a $(\lambda,c)$-quasi--geodesic $q\subset \Gamma$ connecting the endpoints of $\gamma_1$ with the property that $q\not\subset \mathcal{N}_{g(\Gamma)/12}(\gamma_1)$.
\end{lemma}
\begin{proof}
Our goal is to apply \cite[Proposition $5.10$]{ACGH2}. To do this we need to ensure that there is a path $P$ connecting the endpoints of $\gamma_1$ of length at most $\kappa\textrm{diam}(\Gamma)$ such that there exists a point $x\in \gamma_1$ satisfying $d(x,P)\geq \alpha \textrm{diam}(\Gamma)$ for some $\alpha,\kappa$ depending only on $\delta$.

By assumption the path $P=\gamma_l^{-1}\circ \dots \circ \gamma_2^{-1}$ connecting the endpoints of $\gamma_1$ on $C$ has length at most $5\textrm{diam}(\Gamma)$, since it is the concatenation of at most $5$ geodesics.

Suppose for a contradiction that every point $x\in\gamma_1$ satisfies $d(x,P)<\frac{1}{12}g(\Gamma)$.
Let $x_1,x_2$ be points on $\gamma_1$ satisfying $d(x_1,x_2)=\frac{1}{6}g(\Gamma)$ and let $Q_i$ be a geodesic from $x_i$ to $P$ of length $<\frac{1}{12}g(\Gamma)$. Now, $Q_1\cap Q_2=\emptyset$, so there is a simple cycle in $\Gamma$ contained in $Q_1\cup Q_2$ plus the subgeodesic of $\gamma_1$ connecting $x_1$ to $x_2$, and the subpath of $P$ connecting the endpoints of $Q_1$ and $Q_2$. Such a simple cycle has length $< \frac{1}{12}g(\Gamma) + \frac{1}{6}g(\Gamma) + \frac{1}{12}g(\Gamma) + \frac{1}{3}g(\Gamma) < g(\Gamma)$ which contradicts the assumption on the girth of $\Gamma$.
\end{proof}

\begin{proof}[Proof of Theorem $\ref{thm:scancstabletofinpres}$]
Fix some Morse gauge $N$, and choose $N'$ so that $X^{(N)}_e$ is $N'$--stable using Lemma \ref{lem:thirdedgeMorse}. Choose $g$ so that $\varepsilon g>N'(\lambda,c)$ where $\varepsilon,\lambda,c$ are the constants from Lemma \ref{lem:highgirthnotMorse}. Set $\Gamma_N=\bigsqcup\setcon{\Gamma_i}{g(\Gamma_i)\leq g}$. It is clear that $\Gamma_N$ is a finite graph.

For ease of notation we denote the $S$--word metric in $X$ by $d$ and the $S$--word metric in $\fpres{S}{\Gamma(N)}$ by $\overline{d}$.

Consider the map $\overline{\cdot}:X^{(N)}_e\to \fpres{S}{\Gamma_N}$, defined so that $\overline{x}$ is the endpoint of the path $\overline{\gamma_x}$ in $\fpres{S}{\Gamma_N}$ starting at $\overline{e}=1_{\fpres{S}{\Gamma_N}}$ which is equal to $\gamma_x$ as an oriented $S$--labelled path.

It is clear that $d(e,x)=\overline{d}(\overline{e},\overline{x})$ for all $x\in X^{(N)}_e$ and that $d(x,y)\leq \overline{d}(\overline{x},\overline{y})$ for all $x,y\in X^{(N)}_e$.

Suppose for a contradiction that there exist points $x,y\in X^{(N)}_e$ such that $d(x,y)< \overline{d}(\overline{x},\overline{y})$ and let $D$ be a diagram in $X$ with vertices $e,x,y$ whose geodesics sides are all $N'$--Morse.

Now $D$ is not a diagram in $\fpres{S}{\Gamma_N}$ so there is some face $\Pi$ in $D$ which is a translate of a cycle in $\Gamma\setminus\Gamma_N$. By Strebel's classification \ref{thm:Strebel}, $\partial\Pi$ can be written as a concatenation of at most $6$ geodesics $\gamma_1,\dots,\gamma_l$, in such a way that $\gamma_1$ has length at least $\frac{1}{6}\abs{\partial\Pi}$ and is contained in the boundary of $D$.

By Lemma \ref{lem:highgirthnotMorse} $\gamma_1$ is not $N'$--Morse which contradicts that fact that $X^{(N)}_e$ is $N'$--stable.
\end{proof}

Theorems \ref{bthm:smallcanc} and \ref{bthm:Gromovmonster} follow from Theorems \ref{thm:fpresscancasdim2} and \ref{thm:scancstabletofinpres}, together with Osajda's small cancellation labelling of high girth expanders \cite{Osajda}. It follows immediately from \cite{Humexp} that there are $2^{\aleph_0}$ different quasi--isometry classes of finitely generated groups with stable dimension at most $2$ with Cayley graphs which quasi--isometrically contain an expander.

\section{Relatively hyperbolic groups}\label{sec:relhyp}

In this section we prove Theorem \ref{bthm:relhyp}.

\begin{definition}\label{defn:relhyp} Let $G$ be a finitely generated group and let $\mathcal{H}$ be a finite collection of subgroups of $G$. A \textbf{coned--off graph} $\hat{\Gamma}$ of $G$ with respect to $\mathcal{H}$ is a graph obtained from a Cayley graph $\Gamma$ of $G$ by attaching an additional vertex $v_{gH}$ for every left coset of each $H\in\mathcal{H}$ and adding an edge $(v_{gH},g')$ whenever $g'\in gH$.

We say $G$ is \textbf{hyperbolic relative to} $\mathcal{H}$ if the following two conditions are satisfied:
\begin{itemize}
\item A coned--off graph of $G$ is hyperbolic.
\item (Bounded Coset Penetration Property)  Let $\alpha,\beta$ be geodesics in $\hat\Gamma$ with the same endpoints and let $H\in \mathcal{H}$.  Then there exists a constant $c$ such that:
		\begin{enumerate}
		\item if $\alpha\cap gH\neq \emptyset$ but $\beta\cap gH=\emptyset$ for some $g\in G$, then the $\Gamma$-distance between the vertex at which $\alpha$ enters $gH$ and the vertex at which $\alpha$ exits $gH$ is at most $c$, and
		\item If $\alpha\cap gH\neq\emptyset$ and $\beta\cap gH\neq\emptyset$, and $\alpha$ (resp. $\beta$) first enters $gH$ at $\alpha_1$ (resp. $\beta_1$) and last exits $gH_i$ at $\alpha_2$ (resp. $\beta_2$), then $\alpha_j$ and $\beta_j$ are at a $\Gamma$-distance of at most $c$ from each other, for $j=1,2$.
		\end{enumerate}
\end{itemize} 
\end{definition}

\begin{theorem}\label{thm:peripheralsstable} \cite[Lemma 4.15]{drutu-sapir} Let $G$ be hyperbolic relative to $\mathcal{H}$. Fix Cayley graphs $X_G$ of $G$ and $X_H$ of $H$ for each $H\in \mathcal{H}$.
For every Morse gauge $N$ there exists a Morse gauge $N'$ such that

If $x,y\in H$ can be connected by an $N$--Morse geodesic in $X_H$, then they can be connected by an $N'$--Morse geodesic in $X_G$.
\end{theorem}

We now proceed with the proof of Theorem \ref{bthm:relhyp}.

\begin{theorem}\label{thm:upbdRelhyp} Let $G$ be a group which is hyperbolic relative to $\mathcal{H}$. If $\stabdim(H)<\infty$ for all $H\in\mathcal{H}$ then $\max_{H\in\mathcal{H}}\stabdim(H)\leq \stabdim(G)<\infty$.
\end{theorem}

In the case where $asdim(H)<\infty$ for all $H\in\mathcal{H}$, $asdim(G)<\infty$ by \cite{OsinRelhyp}, so the upper bound follows from the simple observation $\stabdim(G)\leq asdim(G)$.

\begin{proof} We begin by proving the lower bound on $\stabdim(G)$. 
Let $x,y\in H$, let $\gamma'$ be an $N$--Morse geodesic from $x$ to $y$ in $X_H$ and let $\gamma$ be any geodesic from $x$ to $y$ in $X_G$. We will show that $\gamma$ is $N'$--Morse, where $N'$ does not depend on $x,y$. Denote the vertices of $\gamma$ by $x_0=x,x_1,\dots,x_d=y$ in order. Let $q$ be any $(K,C)$-quasi--geodesic in $G$ with endpoints $x,y$ and denote the vertices of $q$ by $y_0=x,y_1,\dots,y_l=y$ in order.

By \cite[Lemma 4.15]{drutu-sapir} there exist constants $A,B=B(K,C)$ such that $\gamma\subset N_A(H)$ and $q\subset N_B(H)$. For each vertex $x_i\in\gamma$ (respectively $y_i\in q$) let $h_i\in H$ satisfy $d(x_i,h_i)=d(x_i,H)$ (respectively $h'_i$ satisfy $d(y_i,h'_i)=d(y_i,H)$).

Now as each $X_H$ is undistorted in $X_G$, $h_0,h_1,\dots,h_d$ is a $(\lambda,c)$-quasi--geodesic in $H$ and $h'_0,h'_1,\dots,h'_l$ is a $(\lambda',c')$-quasi--geodesic in $H$. Here $\lambda,\lambda',c,c'$ depend on $K,C,G,H$ but not on $x,y$.

As there is an $N$--Morse geodesic from $x$ to $y$ in $X_H$ we conclude that there exists a $D$ such that $\set{h_0,\dots,h_d}\subset N_D(\set{h'_0,\dots,h'_l})$ and $\set{h'_0,\dots,h'_l}\subset N_D(\set{h_0,\dots,h_d})$. Thus $q\subset N_{A+B+D}(\gamma)$ as required.

Therefore each $H$ is a stable subset of $X_G$.
\medskip

The goal for the upper bound is to determine a uniform bound on the asymptotic dimension of $G^{(N)}_e$.

We recall that by \cite{Hume12} $G$ quasi--isometrically embeds into the product of the coned-off graph $\hat{G}$ and a tree-graded space $\mathcal{T}$ with pieces uniformly quasi--isometric to $H_i$.

Let $\phi$ denote the map from $G$ to $\mathcal{T}$. We claim that for each $N$ there exists some $N'$ such that 
$\phi(G^{(N)}_e)\subseteq \mathcal{T}^{(N')}_{\phi(e)}$.

Given the claim, it is easy to see that $\phi(G^{(N)}_e)$ is contained in a subset $\mathcal{T}^{(N)}$ of $\mathcal{T}$ which is tree-graded and the pieces are uniformly quasi--isometric to $H^{(N')}_e$. By assumption, the asymptotic dimension of the $H^{(N')}_e$ is bounded independent of $N'$.

Thus the spaces $\mathcal{T}^{(N)}$ have uniformly bounded asymptotic dimension. Since the coned-off graph $\hat{G}$ has finite asymptotic dimension \cite{Bell-Fujiwara} we have a uniform bound on the asymptotic dimension of the $G^{(N)}_e$, as required.

\medskip
To prove the claim it suffices to show that if $y\in G^{(N)}_e$ and a geodesic $g=[\phi(e),\phi(y)]$ spends sufficiently long in a piece $P$, then the restriction $\gamma$ of this geodesic to the piece is $N''$--Morse for some uniform $N''$. This follows from the simple fact that a geodesic $\gamma$ in a tree--graded space $X$ is Morse if and only if there is some Morse gauge $N$ such that the intersection of $\gamma$ with each piece is $N$--Morse, in which case $\gamma$ is $N'$--Morse, where $N'$ depends only on $N$ and $X$.

As the natural map from $P$ to $G$ is a quasi--isometric embedding, $\gamma$ defines a $(K,C)$-quasi--geodesic $q_\gamma$ where the constants $K$ and $C$ depend on $G$ but not on $P$ or $\gamma$.

The piece $P$ corresponds to a coset $gH_i$, and if $\gamma$ is sufficiently long then any geodesic $[e,y]$ spends a long time within a fixed neighbourhood of $gH_i$ \cite[Theorem $4.1$]{drutu-sapir}. In particular, $[e,y]$ passes within a fixed distance of each of the endpoints of $q_\gamma$ by BCP. By assumption there is some is $N$--Morse geodesic $[e,y]$, and therefore $q_\gamma$ (and any geodesic with the same endpoints) is contained within a uniformly bounded neighbourhood of $[e,y]$. Thus, such geodesics are $N''$--Morse for some $N''$ which does not depend on $\gamma$. This concludes the proof of the claim.
\end{proof}

{

}


\begin{bibdiv}
\begin{biblist}


\bib{Aougab-Durham-Taylor}{article}{
    AUTHOR={Aougab, Tarik},
    AUTHOR = {Durham, Matthew Gentry},
    AUTHOR= {Taylor, Samuel J.},
     %TITLE = {Divergence, Middle Recurrence, and Pulling Back Subgroup Stability},
	JOURNAL={Contents of a preprint announced in arXiv:1609.06698} 
}


\bib{ACGH1}{article}{
    AUTHOR = {Arzhantseva, Goulnara N.},
    AUTHOR = {Cashen, Christopher H.},
    AUTHOR = {Gruber, Dominik},
    AUTHOR = {Hume, David},
     TITLE = {Characterization of Morse quasi--geodesics via superlinear divergence and sublinear contraction},
	Eprint={1601.01879},
	Journal= {ArXiv e-prints},
    year={2016},
}

\bib{ACGH2}{article}{
    AUTHOR = {Arzhantseva, Goulnara N.},
    AUTHOR = {Cashen, Christopher H.},
    AUTHOR = {Gruber, Dominik},
    AUTHOR = {Hume, David},
     TITLE = {Contracting geodesics in infinitely presented graphical small cancellation groups},
	Eprint={1602.03767},
	Journal= {ArXiv e-prints},
    year={2016},
}

\bib{Arzh-Delz}{article}{
	Author= {Arzhantseva, Goulnara},
	Author= {Delzant, Thomas},
	Title= {Examples of Random Groups},
	Eprint = {http://www.mat.univie.ac.at/~arjantseva/Abs/random.pdf}
}

\bib{behrstock-charney}{article}{,
    AUTHOR = {Behrstock, Jason},
    Author= {Charney, Ruth},
     TITLE = {Divergence and quasimorphisms of right--angled {A}rtin groups},
   JOURNAL = {Math. Ann.},
  FJOURNAL = {Mathematische Annalen},
    VOLUME = {352},
      YEAR = {2012},
    NUMBER = {2},
     PAGES = {339--356},
      ISSN = {0025-5831},
     CODEN = {MAANA},
   MRCLASS = {20F36 (20F65)},
       DOI = {10.1007/s00208-011-0641-8},
       URL = {http://dx.doi.org/10.1007/s00208-011-0641-8},
}

\bib{BHS}{article}{
	AUTHOR={Behrstock, Jason},
	AUTHOR={Hagen, Mark},
	AUTHOR={Sisto, Alessandro},
	Eprint={1512.06071},
	Journal= {ArXiv e-prints},
	Title= {Asymptotic dimension and small-cancellation for hierarchically hyperbolic spaces and groups},
	Year= {2015},
	}

\bib{Bell-Fujiwara}{article}{
	author = {Bell, Gregory},
	author = {Fujiwara, Koji},
	title = {The asymptotic dimension of a curve graph is finite},
   JOURNAL = {J. Lond. Math. Soc. (2)},
  FJOURNAL = {Journal of the London Mathematical Society. Second Series},
    VOLUME = {77},
      YEAR = {2008},
    NUMBER = {1},
     PAGES = {33--50},
	}

\bib{BestvinaQList}{article}{
	Author = {Bestvina, Mladen},
	Title= {Questions in Geometric Group Theory},
	Eprint={http://www.math.utah.edu/~bestvina/eprints/questions-updated.pdf}}

\bib{Bestvina-Bromberg}{article}{
	%Archiveprefix = {arXiv},
    AUTHOR = {Bestvina, Mladen},
    AUTHOR = {Bromberg, Ken},
	Eprint = {1508.04832},
	Journal = {ArXiv e-prints},
	%Keywords = {Mathematics - Geometric Topology, 20F65, 20F67},
	%Primaryclass = {math.GT},
	Title = {On the asymptotic dimension of the curve complex},
	Year = {2015}
	}


\bib{BBF}{article}{
    AUTHOR = {Bestvina, Mladen},
    AUTHOR = {Bromberg, Ken},
    AUTHOR={Fujiwara, Koji},
     TITLE = {Constructing group actions on quasi-trees and applications to
              mapping class groups},
   JOURNAL = {Publ. Math. Inst. Hautes \'Etudes Sci.},
  FJOURNAL = {Publications Math\'ematiques. Institut de Hautes \'Etudes
              Scientifiques},
    VOLUME = {122},
      YEAR = {2015},
     PAGES = {1--64},
      ISSN = {0073-8301},
       DOI = {10.1007/s10240-014-0067-4},
       URL = {http://dx.doi.org/10.1007/s10240-014-0067-4},
}

\bib{BestMess}{article}{
   author = {Bestvina, Mladen},
   author = {Mess, Geoffrey},
   Fjournal = {Journal of the American Mathematical Society},
  Journal = {J. Amer. Math. Soc.},
  Number = {3},
	Pages = {469--481},
    title = {The boundary of negatively curved groups},
    Volume = {4},
	Year = {1991},
}
\bib{BonkSchramm}{article}{
    AUTHOR = {Bonk, Mario},
    AUTHOR={Schramm, Oded},
     TITLE = {Embeddings of {G}romov hyperbolic spaces},
   JOURNAL = {Geom. Funct. Anal.},
  FJOURNAL = {Geometric and Functional Analysis},
    VOLUME = {10},
      YEAR = {2000},
    NUMBER = {2},
     PAGES = {266--306},
      ISSN = {1016-443X},
     CODEN = {GFANFB},
   MRCLASS = {53C23 (54E40 57M07)},
  MRNUMBER = {1771428},
MRREVIEWER = {Michel Coornaert},
%       DOI = {10.1007/s000390050009},
%       URL = {http://dx.doi.org/10.1007/s000390050009},
}


\bib{bh}{book}{
	Address = {Berlin},
	Author = {Bridson, Martin R.},
	author = {Haefliger, Andr{\'e}},
	Pages = {xxii+643},
	Publisher = {Springer-Verlag},
	Series = {Grundlehren der Mathematischen Wissenschaften [Fundamental Principles of Mathematical Sciences]},
	Title = {Metric spaces of non-positive curvature},
	Volume = {319},
	Year = {1999},
	}

	
\bib{Buycapdim}{article}{
    AUTHOR = {Buyalo, Sergei V.},
     TITLE = {Asymptotic dimension of a hyperbolic space and the capacity
              dimension of its boundary at infinity},
   JOURNAL = {Algebra i Analiz},
    VOLUME = {17},
      YEAR = {2005},
    NUMBER = {2},
     PAGES = {70--95},
}

\bib{BuyLeb}{article}{
    AUTHOR = {Buyalo, Sergei V.},
    AUTHOR = {Lebedeva, Nina D.},
     TITLE = {Dimensions of locally and asymptotically self-similar spaces},
   JOURNAL = {Algebra i Analiz},
    VOLUME = {19},
      YEAR = {2007},
    NUMBER = {1},
     PAGES = {60--92},
     }
     
\bib{BuyBook}{book}{
    AUTHOR = {Buyalo, Sergei},
    AUTHOR= {Schroeder, Viktor},
     TITLE = {Elements of asymptotic geometry},
    SERIES = {EMS Monographs in Mathematics},
 PUBLISHER = {European Mathematical Society (EMS), Z\"urich},
      YEAR = {2007},
     PAGES = {xii+200},
}

\bib{Cashen16}{article}{
    AUTHOR = {Cashen, Christopher H.},
    TITLE = {Quasi-isometries need not induce homeomorphisms of contracting boundaries with the {G}romov product topology},
   JOURNAL = {Anal. Geom. Metr. Spaces},
  FJOURNAL = {Analysis and Geometry in Metric Spaces},
    VOLUME = {4},
      YEAR = {2016},
     PAGES = {278--281},  
}
	

     
	
\bib{charney-sultan}{article}{,
    AUTHOR = {Charney, Ruth},
    Author=  {Sultan, Harold},
     TITLE = {Contracting boundaries of {$\rm CAT(0)$} spaces},
   JOURNAL = {J. Topol.},
  FJOURNAL = {Journal of Topology},
    VOLUME = {8},
      YEAR = {2015},
    NUMBER = {1},
     PAGES = {93--117},
   MRCLASS = {20F65},
       URL = {http://doi.org/10.1112/jtopol/jtu017},
}
	
\bib{Cordes15}{article}{
	%Archiveprefix = {arXiv},
	Author = {{Cordes}, Matthew},
	%Eprint = {1502.04376},
	Journal = {To appear in Groups, Geometry, and Dynamics},
	%Keywords = {Mathematics - Geometric Topology, 20F65, 20F67},
	%Primaryclass = {math.GT},
	Title = {{Morse boundaries of proper geodesic metric spaces}},
	%Year = {2015},
	}

\bib{CordesHume_relativehyp}{article}{
	Author = {{Cordes}, Matthew},
	Author = {{Hume}, David},
	Eprint = {1609.05154},
	Journal = {ArXiv e-prints},
	%Keywords = {Mathematics - Geometric Topology, 20F65, 20F67},
	%Primaryclass = {math.GT},
	Title = {{Relatively hyperbolic groups with fixed peripherals}},
	Year = {2016},
	}

\bib{CroKle}{article}{
	Author = {Croke, Christopher B.},
	Author = {Kleiner, Bruce},
	Coden = {TPLGAF},
	Fjournal = {Topology. An International Journal of Mathematics},
	Issn = {0040-9383},
	Journal = {Topology},
	Mrreviewer = {Joseph E. Borzellino},
	Number = {3},
	Pages = {549--556},
	Title = {Spaces with nonpositive curvature and their ideal boundaries},
	Url = {http://dx.doi.org/10.1016/S0040-9383(99)00016-6},
	Volume = {39},
	Year = {2000},
	}


\bib{Dowdall:2014ab}{article}{
	Author = {Dowdall, Spencer},
	Author= {{Kent IV}, Richard P.},
	Author= {Leininger, Christopher J.},
	Fjournal = {Groups, Geometry, and Dynamics},
	Issn = {1661-7207},
	Journal = {Groups Geom. Dyn.},
	Number = {4},
	Pages = {1247--1282},
	Title = {Pseudo-{A}nosov subgroups of fibered 3-manifold groups},
	Volume = {8},
	Year = {2014},
}


\bib{Kapovich-Drutu}{book}{
	author={Dru{\c{t}}u, Cornelia},
	Author = {Kapovich, Michael},
	Title = {Lectures on Geometric Group Theory},
	Publisher = {preprint},
}
	
\bib{Drutu-Mozes-Sapir}{article} {
    AUTHOR = {Dru{\c{t}}u, Cornelia},
    AUTHOR={Mozes, Shahar},
    AUTHOR= {Sapir, Mark},
     TITLE = {Divergence in lattices in semisimple {L}ie groups and graphs
              of groups},
   JOURNAL = {Trans. Amer. Math. Soc.},
  FJOURNAL = {Transactions of the American Mathematical Society},
    VOLUME = {362},
      YEAR = {2010},
    NUMBER = {5},
     PAGES = {2451--2505},
      ISSN = {0002-9947},
       DOI = {10.1090/S0002-9947-09-04882-X},
       URL = {http://dx.doi.org/10.1090/S0002-9947-09-04882-X},
}

	
	
\bib{drutu-sapir}{article}{
	author={Dru{\c{t}}u, Cornelia},
	author={Sapir, Mark},
	title={Tree-graded spaces and asymptotic cones of groups},
	Journal={Topology},
	volume={44(5)},
	pages={959--1058} ,
	year={2005},
	note={With an appendix by Denis Osin and Mark Sapir},
	}

\bib{Durham-Taylor}{article}{
    AUTHOR = {Durham, Matthew Gentry},
    AUTHOR= {Taylor, Samuel J.},
     TITLE = {Convex cocompactness and stability in mapping class groups},
   JOURNAL = {Algebr. Geom. Topol.},
  FJOURNAL = {Algebraic \& Geometric Topology},
    VOLUME = {15},
      YEAR = {2015},
    NUMBER = {5},
     PAGES = {2839--2859},
      ISSN = {1472-2747},
   MRCLASS = {20F65 (30F60 57M07)},
  MRNUMBER = {3426695},
}


		

	

\bib{Farb-Mosher}{article}{
	Author = {Farb, Benson},
	Author= {Mosher, Lee},
	Fjournal = {Geometry and Topology},
	Journal = {Geom. Topol.},
	Mrclass = {20F65 (20F67 57M07 57S25)},
	Pages = {91--152 (electronic)},
	Title = {Convex cocompact subgroups of mapping class groups},
	Volume = {6},
	Year = {2002},}


\bib{GhysdlH}{incollection}{
    AUTHOR = {Ghys, {\'E}tienne},
    author = {de la Harpe, Pierre},
     TITLE = {quasi--isom\'etries et quasi-g\'eod\'esiques},
 BOOKTITLE = {Sur les groupes hyperboliques d'apr\`es {M}ikhael {G}romov
              ({B}ern, 1988)},
    SERIES = {Progr. Math.},
    VOLUME = {83},
     PAGES = {79--102},
 PUBLISHER = {Birkh\"auser Boston, Boston, MA},
      YEAR = {1990},
}



\bib{Gromov87}{incollection}{,
	Author = {Gromov, Misha},
	Booktitle = {Essays in group theory},
	Pages = {75--263},
	Publisher = {Springer, New York},
	Series = {Math. Sci. Res. Inst. Publ.},
	Title = {Hyperbolic groups},
	Url = {http://dx.doi.org/10.1007/978-1-4613-9586-7_3},
	Volume = {8},
	Year = {1987},
	}
	
\bib{Gromov03}{article}{,
    AUTHOR = {Gromov, Misha},
     TITLE = {Random walk in random groups},
   JOURNAL = {Geom. Funct. Anal.},
  FJOURNAL = {Geometric and Functional Analysis},
    VOLUME = {13},
      YEAR = {2003},
    NUMBER = {1},
     PAGES = {73--146},
      ISSN = {1016-443X},
     CODEN = {GFANFB},
   MRCLASS = {20F65 (20F67 20P05 60G50)},
  MRNUMBER = {1978492},
MRREVIEWER = {Thomas Delzant},
       DOI = {10.1007/s000390300002},
       URL = {http://dx.doi.org/10.1007/s000390300002},
}


\bib{Gruber-graphicalC(6)C(7)}{article}{,
	author = {Gruber, Dominik},
	title = {Groups with graphical {$C(6)$} and {$C(7)$} small cancellation presentations},
 journal              = {Trans. Amer. Math. Soc.},
 number               = {3},
 pages                = {2051--2078},
 volume               = {367},
 year                 = {2015},
	}

\bib{Gruber_personal}{article}{,
	author = {Gruber, Dominik},
	title = {Personal communication},
	}

	
\bib{hagen:2014aa}{article}{,
	Author = {Hagen, Mark F.},
	Doi = {10.1112/jtopol/jtt027},
	Fjournal = {Journal of Topology},
	Issn = {1753-8416},
	Journal = {J. Topol.},
	Mrclass = {20F65 (20F67)},
	Number = {2},
	Pages = {385--418},
	Title = {Weak hyperbolicity of cube complexes and quasi-arboreal groups},
	Volume = {7},
	Year = {2014},
	}

\bib{HigLafSka02}{article}{,
    AUTHOR = {Higson, Nigel},
    Author= {Lafforgue, Vincent},
    Author= {Skandalis, Georges},
     TITLE = {Counterexamples to the {B}aum-{C}onnes conjecture},
   JOURNAL = {Geom. Funct. Anal.},
  FJOURNAL = {Geometric and Functional Analysis},
    VOLUME = {12},
      YEAR = {2002},
    NUMBER = {2},
     PAGES = {330--354},
      ISSN = {1016-443X},
     CODEN = {GFANFB},
}

\bib{Hume12}{article}{
    AUTHOR = {Hume, David},
     TITLE = {{Embedding Mapping Class Groups into a finite product of trees}},
      YEAR = {2012},
    Eprint = {1207.2132},
   Journal = {To appear in Groups Geom. Dyn. ArXiv e-prints},
}

\bib{Humexp}{article}{
	Author = {Hume, David},
	Eprint = {1410.0246},
	Year = {2014},
	Journal = {To appear in Fund. Math. ArXiv e-prints},
	Title = {A continuum of expanders},
	}


\bib{Kent:2009aa}{article}{
	Author = {{Kent IV}, Richard P.},
	Author={Leininger, Christopher J.},
	Author= {Schleimer, Saul},
	Coden = {JRMAA8},
	Doi = {10.1515/CRELLE.2009.087},
	Fjournal = {Journal f\"ur die Reine und Angewandte Mathematik. [Crelle's Journal]},
	Issn = {0075-4102},
	Journal = {J. Reine Angew. Math.},
	Mrreviewer = {Javier Aramayona},
	Pages = {1--21},
	Title = {Trees and mapping class groups},
	Url = {http://dx.doi.org/10.1515/CRELLE.2009.087},
	Volume = {637},
	Year = {2009},
}

\bib{koberda:2014}{article}{
	Author = { Koberda, Thomas},
	Author = {Mangahas, Johanna},
	Author= {Taylor, Samuel J.},
	Eprint = {1412.3663},
	Journal = {ArXiv e-prints},
	Title = {The geometry of purely loxodromic subgroups of right--angled Artin groups},
	Year = {2014}}


\bib{leininger:2014aa}{article}{
	Author = {Leininger, Christopher},
	Author= {Schleimer, Saul},
	Doi = {10.4171/JEMS/495},
	Fjournal = {Journal of the European Mathematical Society (JEMS)},
	Issn = {1435-9855},
	Journal = {J. Eur. Math. Soc. (JEMS)},
	Mrclass = {Preliminary Data},
	Number = {12},
	Pages = {2669--2692},
	Title = {Hyperbolic spaces in {T}eichm\"uller spaces},
	Volume = {16},
	Year = {2014},
}


\bib{LS01}{book}{
  title = {Combinatorial group theory},
  publisher = {Springer-Verlag},
  year = {2001},
  author = {Lyndon, Roger C.},
  author = {Schupp, Paul E.},
  pages = {xiv+339},
  series = {Classics in Mathematics},
  address = {Berlin},
  note = {Reprint of the 1977 edition},
  }


\bib{MackSis}{article}{
    AUTHOR = {Mackay, John M.},
    AUTHOR = {Sisto, Alessandro},
     TITLE = {Embedding relatively hyperbolic groups in products of trees},
   JOURNAL = {Algebr. Geom. Topol.},
    VOLUME = {13},
      YEAR = {2013},
    NUMBER = {4},
     PAGES = {2261--2282},
}
\bib{Masur:1999aa}{article}{,
	Author = {Masur, Howard A.},
	Author= {Minsky, Yair N.},
	Journal = {Invent. Math.},
	Number = {1},
	Pages = {103--149},
	Title = {Geometry of the complex of curves. {I}. {H}yperbolicity},
	Volume = {138},
	Year = {1999}}
	
\bib{Masur:2000aa}{article}{,
	Author = {Masur, Howard A.},
	Author={Minsky, Yair N.},
	Coden = {GFANFB},
	Fjournal = {Geometric and Functional Analysis},
	Issn = {1016-443X},
	Journal = {Geom. Funct. Anal.},
	Number = {4},
	Pages = {902--974},
	Title = {Geometry of the complex of curves. {II}. {H}ierarchical structure},
	Volume = {10},
	Year = {2000}}
	
\bib{Min:2011aa}{article}{
	Author = {Min, Honglin},
	Doi = {10.2140/agt.2011.11.449},
	Fjournal = {Algebraic \& Geometric Topology},
	Issn = {1472-2747},
	Journal = {Algebr. Geom. Topol.},
	Mrclass = {20F67 (20F28 20F65 57M07)},
	Mrnumber = {2783234},
	Mrreviewer = {Igor Belegradek},
	Number = {1},
	Pages = {449--476},
	Title = {Hyperbolic graphs of surface groups},
	Url = {http://dx.doi.org/10.2140/agt.2011.11.449},
	Volume = {11},
	Year = {2011},
}

\bib{Morse}{article}{,
    AUTHOR = {Morse, Harold Marston},
     TITLE = {A fundamental class of geodesics on any closed surface of
              genus greater than one},
   JOURNAL = {Trans. Amer. Math. Soc.},
  FJOURNAL = {Transactions of the American Mathematical Society},
    VOLUME = {26},
      YEAR = {1924},
    NUMBER = {1},
     PAGES = {25--60},
      ISSN = {0002-9947},

}

\bib{Murray}{article}{
    AUTHOR = {Murray, Devin},
    Eprint = {1509.09314},
     TITLE = {Topology and Dynamics of the Contracting Boundary of Cocompact CAT(0) Spaces},
   Journal = {ArXiv e-prints},
      Year = {2015}
}	

\bib{Ollivier}{article}{
 author               = {Ollivier, Yann},
 journal              = {Bull. Belg. Math. Soc. Simon Stevin},
 number               = {1},
 pages                = {75--89},
 title                = {On a small cancellation theorem of {G}romov},
 url                  = {http://projecteuclid.org/euclid.bbms/1148059334},
 volume               = {13},
 year                 = {2006},
 }


\bib{Osajda}{article}{
    AUTHOR = {Osajda, Damian},
     TITLE = {Small cancellation labellings of some infinite graphs and applications},
      YEAR = {2014},
    Eprint = {1406.5015},
   Journal = {ArXiv e-prints},
}



\bib{OsinRelhyp}{article}{
    AUTHOR = {Osin, Denis V.},
     TITLE = {Relatively hyperbolic groups: intrinsic geometry, algebraic
              properties, and algorithmic problems},
   JOURNAL = {Mem. Amer. Math. Soc.},
    VOLUME = {179},
      YEAR = {2006},
    NUMBER = {843},
     PAGES = {vi+100},
}

\bib{OsinAclynhyp}{article}{
	Author = {Osin, Denis V.},
	Title={Acylindrically hyperbolic groups}  ,
   Journal = {Trans. Amer. Math. Soc.},
   Volume = {368},
   Number = {2},
   pages = {851--888},
      Year = {2016}}


\bib{PansuConf89}{article}{
    AUTHOR = {Pansu, Pierre},
     TITLE = {Dimension conforme et sph\`ere \`a l'infini des vari\'et\'es
              \`a courbure n\'egative},
   JOURNAL = {Ann. Acad. Sci. Fenn. Ser. A I Math.},
    VOLUME = {14},
      YEAR = {1989},
    NUMBER = {2},
     PAGES = {177--212},
}


\bib{sisto:2016aa}{article}{
	Author = {Sisto, Alessandro},
	Issn = {1432-1823},
	Journal = {Mathematische Zeitschrift},
	Pages = {1--10},
	Title = {Quasi--convexity of hyperbolically embedded subgroups},
	Year = {2016}}


\bib{Str90}{article}{
 author               = {Strebel, Ralph},
 booktitle            = {Sur les groupes hyperboliques d'apr{\`e}s {M}ikhael {G}romov ({B}ern, 1988)},
 date-added           = {2014-07-19 04:51:35 +0000},
 date-modified        = {2014-07-19 04:51:47 +0000},
 mrclass              = {57M05 (20F06 53C23)},
 mrnumber             = {1086661},
 pages                = {227--273},
 publisher            = {Birkh{\"a}user Boston, Boston, MA},
 series               = {Progr. Math.},
 title                = {Appendix. {S}mall cancellation groups},
 volume               = {83},
 year                 = {1990},
 }

\end{biblist}
\end{bibdiv}
\end{document}